\documentclass{article}
\pdfoutput=1
\usepackage{graphicx} 
\usepackage{authblk}
\usepackage{array}
\usepackage{xcolor}

\title{Entropy of Random Geometric Graphs in High and Low Dimensions}
\author[1]{Oliver Baker}
\author[1]{Carl P. Dettmann}
\affil[1]{School of Mathematics, University of Bristol}
\date{\today}

\usepackage[a4paper,
            total={126mm, 195mm}]{geometry}
            
\usepackage{graphicx} 
\usepackage{hyperref}
\usepackage[nottoc,numbib]{tocbibind}
\usepackage{amsmath}
\usepackage{amsthm}
\usepackage{amssymb}
\usepackage{amsfonts}
\usepackage{breqn}
\usepackage{pdflscape}
\usepackage{algorithm}
\usepackage{algpseudocode}
\usepackage{multicol}
\usepackage{blindtext}
\usepackage[]{caption}

\DeclareMathOperator*{\argmax}{argmax}

\newtheorem{definition}{Definition}
\newtheorem{lemma}{Lemma}
\newtheorem{theorem}{Theorem}
\newtheorem{corollary}{Corollary}
\newtheorem{proposition}{Proposition}

\newtheorem{conjecture}{Conjecture}
\newtheorem{remark}{Remark}

\newcommand{\tti}{\rightarrow\infty}
\newcommand{\graph}{\mathcal{G}}
\newcommand{\bigO}[1]{\mathcal{O}\left(#1\right)}

\newcommand{\prob}[1]{\mathbb{P}\left(#1\right)}

\newcommand{\ttz}[0]{\rightarrow 0}

\begin{document}

\maketitle

\begin{abstract}
    We use a multivariate central limit theorem (CLT) to study the distribution of random geometric graphs (RGGs) on the cube and torus in the high-dimensional limit with general node distributions. We find that the distribution of RGGs on the torus converges to the Erd\H os-R\'enyi (ER) ensemble when the nodes are uniformly distributed, but that the distribution for RGGs with non-uniformly distributed nodes on the torus, and for RGGs with any distribution of nodes with kurtosis greater than 1 on the cube is different. In these cases, the distribution has a lower maximum entropy than the ER ensemble, but is still symmetric. Soft RGGs in either geometry converge to the ER ensemble. An Edgeworth correction to the CLT is then developed to derive the $\bigO{d^{-\frac{1}{2}}}$ sub-leading term of the Shannon entropy of RGGs in dimension for both geometries. We also provide numerical approximations of maximum entropy in low-dimensional hard and soft RGGs, and calculate exactly the entropy of hard RGGs with 3 nodes in the one-dimensional cube and torus. 
\end{abstract}

\newpage

\section{Introduction}
The maximum entropy principle states that when modelling a random process, one should choose the distribution with the maximum entropy, since this assumes the least about the process \cite{jaynes1982rationale}. This means that understanding the entropy of a distribution on some model will help us to understand the typical realisation of the model. In the current age, we are surrounded by processes that permit a representation as networks, including those in biology \cite{west2012differential}, social sciences \cite{marin2011social}, telecommunications \cite{haenggi2009stochastic} and data science \cite{zhang2022joint}. \newline

\noindent The dimension of networks in data science in particular is ever increasing, where in the modern age fields such as computer vision and large language models require datasets with a huge number of inputs. For example one of the most famous baseline datasets for image recognition, MNIST has 784 features \cite{lecun1998gradient}, and much more recently, the embedding space of vectors in GPT3 is 12288-dimensional \cite{NEURIPS2020_1457c0d6}. It is also commonly observed that many high dimensional datasets exhibit a latent embedding, which can be interpreted as some notion of geometry. For example, in image recognition, for a list of pixel colours, there is a geometric interpretation where certain pixels are close to each other in the image. Therefore for algorithms such as nearest neighbour clustering which rely on proximity graphs, it is crucial to understand the geometry of high dimensional spaces, and how we construct networks on them in the most appropriate manner \cite{mittal2019clustering}. \newline

\noindent The concept of graph entropy was first introduced by Rashevsky in 1955 \cite{rashevsky1955life}, and has since become a well-researched area. The question of determining the uncertainty or complexity of a random network has inspired multiple definitions of the concept of graph entropy, including Shannon entropy  \cite{coon2018entropy}, Von Neumann entropy \cite{passerini2009quantifying}, Gibbs entropy \cite{anand2010gibbs} and Kolmogorov complexity \cite{mowshowitz2012entropy}. In this work we will focus on the Shannon entropy formalism, which calculates the entropy in the information-theoretic sense of a distribution of labelled random graphs. \newline

\noindent The Random Geometric Graph (RGG), first introduced by Gilbert in 1961 \cite{gilbert1961random} (under the name of `random planar networks') is an example of such a proximity graph. The RGG is constructed by randomly distributing points in a domain, and connecting those which lie within a specified distance of each other. An extension to the model is the Soft Random Geometric Graph (SRGG), which adds an extra layer of randomness by connecting points with a probability that depends on their pairwise distance. Both models are very well studied from the perspective of, for example, percolation, \cite{penrose2003random} connectivity \cite{penrose2016connectivity}, and centrality measures \cite{giles2015betweenness} which can be seen as \textit{graph-specific} measures, in the sense that we can measure the betweenness centrality of an individual graph, and we can say whether an individual graph is connected. \newline

\noindent Less research however has been directed towards \textit{ensemble-wide} properties of SRGGs, for example we cannot measure the Shannon entropy of an individual graph, but rather the entropy of a distribution induced on a set of graphs. Currently, we are able to obtain bounds on the entropy of an SRGG ensemble \cite{coon2018entropy, badiu2018distribution}, and we can characterise and maximise the entropy of spatial network ensembles given exact and expected constraints, such as the degree distribution \cite{bianconi2021information}. We are unable to analytically maximise the entropy of an unconstrained ensemble, except when conditioned on the location of nodes \cite{coon2018entropy}. In \cite{paton2022entropy}, the important topic of whether we should discuss the entropy of labelled or unlabelled networks is discussed. We focus on the labelled case here, and leave the unlabelled case as future research. \newline

\noindent  High dimensional random geometric graphs have also been of interest in the literature. Devroye et al. showed in 2011 that the random geometric graph embedded on the surface of the unit sphere approaches the Erd\H os-R\'enyi (ER) graph (denoted as $G(n,p)$) in the high-dimensional limit \cite{devroye2011high}, whereas more recently Erba et al. \cite{erba2020random} have shown that when we embed the RGG on the hypercube, the expected clique number differs from the ER limit. There has been a large body of work in testing for geometry in RGGs. We draw particular attention to \cite{bangachev2023detection} which provides testing thresholds for RGGs on the torus using signed triangles, \cite{liu2022testing} which uses RGGs on the surface of the unit sphere with belief propagation algorithms, \cite{friedrich2023cliques} which tests for cliques in geometric inhomogeneous random graphs in the high-dimensional torus, and, very recently, \cite{baguley2025testing} which uses an Edgeworth expansion for geometry detection, as well as to investigate spectral properties of the RGGs on the torus. We direct interested readers to references within these articles for further information. All of these works consider only uniform distributions of nodes, and show that the RGG ensemble with uniformly distributed nodes converges to the ER ensemble. For other distributions, we can give examples such as \cite{brennan2024threshold} which studies the \textit{anisotropic} RGG with a variable $n$, and the book of Penrose \cite{penrose2003random} that deals with the connectivity threshold in  $\mathbb{R}^d$  for Gaussian distributed points. We study random geometric graphs in the regime where we have a fixed number of nodes, connection range dependent only on the dimension of the domain, and in both the high-dimensional cube and torus with more general node distributions, requiring only independence in each coordinate.

In this work, we make progress in the knowledge of the entropy of random geometric graph ensembles in low and high dimensions by
\begin{enumerate}
    \item Analytically calculating the entropy of 3-node hard RGG ensembles in the one dimensional cube and torus,
    \item Performing numerical simulations to investigate the maximum entropy of SRGGs in low dimensions,
    \item Proving that distributing nodes non-uniformly in $\mathbb{T}^d$ causes the hard RGG to deviate from the ER limit as $d\tti$, 
    \item Proving that the hard RGG ensemble in $[0,1]^d$ never converges to the ER limit for ensembles where the nodes are i.i.d. and have kurtosis greater than 1 in each coordinate, 
    \item  Deriving the scaling of entropy in dimension in both geometries using an Edgeworth correction to the CLT.  
\end{enumerate} 
The rest of the paper is structured as follows. Section 2 outlines the necessary background knowledge for the paper. Section 3 contains the exact calculation of RGG entropy in $[0,1]$ and the 1D torus $\mathbb{T}$ for $n=3$, and numerical simulations for SRGGs in low dimensions. Section 4 contains the multivariate central limit theorems for cubes and tori, the proofs of convergence (or non-convergence) to $G(n,p)$ for different node distributions for hard and soft RGGs, and numerical evidence that the entropy has a unique maximum in the connection range parameter. In Section 5 we develop an Edgeworth correction to the CLTs for a uniform distribution of nodes, and finally conclude the work in Section 6. 

\section{Preliminaries}
In this paper, we will work in two geometric settings. The first is the $d$-dimensional unit cube $[0,1]^d$ with the Euclidean norm, $\|x\|$. The second is the $d$-dimensional unit torus $\mathbb{T}^d=(\mathbb{R}\setminus\mathbb{Z})^d$. We define the distance between two points $x = (x_1,...x_d), y=(y_1,...y_d) \in \mathbb{T}^d$ as 
\begin{equation}
    \label{eq:torus_metric}
    \rho_T(x,y) = \sqrt{\sum_{i=1}^d \min(|x_i-y_i|, 1-|x_i-y_i|)^2}
\end{equation}  
We will construct a \textit{random geometric graph} in each geometry, and analyse the \textit{entropy of the ensemble} as $d\tti$. We define these terms below.
\begin{definition}
    A (hard) Random Geometric Graph (RGG) $G = (V,E)$ on the metric space $(\Omega, \rho)$ is constructed as follows. Distribute $n$ points independently and at random at points $X_1,...,X_n$ in $\Omega$ according to a probability measure $\nu$ to form the vertex set $V = \{1,...,n\}$. Then, form the edge set $E$ by including each edge $(i,j)$ if and only if $\rho(X_i, X_j) \leq r_0$ where $r_0$ is the `connection radius'.
\end{definition}
\begin{definition}
    A Soft Random Geometric Graph (SRGG) $G = (V,E)$ on the metric space $(\Omega, \rho)$ is constructed as follows. Distribute $n$ points independently and at random at points $X_1,...,X_n$ in $\Omega$ according to a probability measure $\nu$ to form the vertex set $V = \{1,...,n\}$. Then, form the edge set $E$ by including each edge $(i,j)$ independently with probability $p(\rho(X_i, X_j)/r_0)$ where $r_0$ is the `connection radius'. Here, the `connection function' $p : \mathbb{R}_{\geq 0} \rightarrow [0,1]$ is a non-increasing function.
\end{definition}
When considering the limit of high dimension ($d\tti$) in either of these cases, we take $r_0$ to be a function of $d$, but \textit{not} of the number of nodes $n$, which we keep fixed. It is clear that a hard RGG is a special case of a SRGG with connection function $p(r/r_0) = \mathbb{I}(r\leq r_0)$ where $\mathbb{I}$ is the indicator function. Often the term RGG is used to refer to hard RGGs specifically, but in the remainder of this paper, we will always state whether the RGG in question is formed using a hard or soft connection function to avoid confusion. A common connection function used in the wireless communications literature is the \textit{Rayleigh Fading} connection function, given by \cite{dettmann2016random}
\begin{equation}
    p(r/r_0) = \exp\left(-\left(\frac{r}{r_0}\right)^\eta\right)
\end{equation}
where $\eta$ is known as the \textit{path loss exponent}, and is general considered to be between 2 and 6 in practical applications \cite{haenggi2009stochastic}. When $\eta=1$ the SRGG is known as the Waxman graph \cite{waxman1988routing}. When $\eta\tti$, we recover the hard RGG formalism. We will use this connection function in Section \ref{low_dim}. 
\begin{definition}
    An `ensemble` $\graph$ of (S)RGGs is the set of all possible (S)RGGs that can be constructed with a fixed $r_0$, $n$, $\Omega$, $\nu$ and $\rho$ (and $p$ for the SRGG). We say a sequence of graph ensembles $\{\graph_{d}\}$ (with parameters dependent on $d$) equipped with probability measures $\mathbb{P}_d$ converges in distribution to a graph ensemble $\graph$ with probability measure $\mathbb{P}$ as $d\tti$ if for all labelled graphs $g$ with $n$ nodes
    \begin{equation}
        \mathbb{P}_d(g) \rightarrow \prob{g}
    \end{equation}
\end{definition}
Finally, we define the \textit{entropy} of the ensemble.
\begin{definition}
    The `entropy' of an ensemble $\graph$ equipped with a probability measure $\mathbb{P}$ is given by
    \begin{equation}
        H(\graph) := -\sum_{g \in \graph} \prob{g}\log \prob{g}
    \end{equation}
    where the logarithm is to base 2.
\end{definition}
In general, the quantity $\prob{g}$ is intractable to compute. If we let $f_{\vec{R}}(\vec{r})$ be the joint density of the pairwise distances $r_{ij}$ between $n$ independently placed points in $\Omega$ according to $\nu$, then
\begin{equation}
    \label{eq:graph_probability}
    \prob{g} = \int_{[0,D]^{\binom{n}{2}}} f_{\vec{R}}(\vec{r})\prod_{1\leq i < j \leq n} p(r_{ij}/r_0)^{x_{ij}}(1-p(r_{ij}/r_0))^{1-x_{ij}} d\vec{r}
\end{equation}
where $D$ is the diameter (maximum distance between two points) of $\Omega$, $x_{ij}$ is the indicator variable for when $(i,j) \in E$ (the edge set of $g$), and $d\vec{r} = \prod_{1\leq i<j \leq n} dr_{ij}$. Note that the density is well defined only when $n < d+2$. For example, if $d=1$ and $n=3$, then the largest distance is the sum of the smaller two. We will write $f(r)$ for the density of one pair distance. 
\begin{definition}
    The `average connection probability' of an ensemble $\mathcal{G}$ is defined as 
    \begin{equation}
        \bar{p} = \bar{p}(r_0) := \int_0^D f(r)p(r/r_0) dr
    \end{equation}
    If $\graph$ is an ensemble of hard RGGs, this reduces to
    \begin{equation}
        \bar{p}(r_0) = \int_0^{r_0} f(r)dr
    \end{equation}
    Define the `entropy maximising connection range' as 
    \begin{equation}
        \hat{r}_0 = \argmax_{r_0} H(\mathcal{G})
    \end{equation}
    and the `entropy maximising average connection probability' as $\bar{p}_{max} = \bar{p}(\hat{r}_0)$. Note that we assume the entropy has a unique maximum in $r_0$, which is the case for all examples we consider in this paper.
\end{definition}
We note that if for an SRGG ensemble $\graph$, $p(r/r_0)$ is independent of $r$, then the edges of the SRGG are independent, and the ensemble is equivalent to the Erd\H os-R\'enyi (ER) ensemble $G(n, \bar{p})$. The entropy of $G(n,p)$ is $H(G(n,p)) = \binom{n}{2}h_2(p)$, maximised at $p = \frac{1}{2}$ to give $H(G(n,1/2)) = \binom{n}{2}$. Here $h_2(p)$ is the binary entropy function,
\begin{equation}
    h_2(p) = -p\log p - (1-p)\log (1-p)
\end{equation}
When $p=1/2$, there is a uniform distribution of graphs in $G(n,p)$, and in turn, a binomial edge distribution. This ensemble reaches the maximum possible entropy of any graph ensemble on $n$ nodes.
\section{RGG Entropy for Small $n$ and $d$}
\label{low_dim}
In low dimensions, we have an explicit expression for the pair distance expression $f(r)$ if we assume a \textit{uniform} distribution of nodes in $\Omega$. This allows us to derive exact entropy and $\bar{p}_{max}$ results for hard RGGs. For the simplest non-trivial case, ($d=1$ and $n=3$), we can obtain exact expressions for the ensemble entropy. This is the first time that the exact ensemble \textit{entropy} has been obtained for any (S)RGG ensemble. However, we draw particular attention to the work of Badiu and Coon \cite{badiu2018distribution}, who derived an exact expression for the 3 node distance density $f_{\vec{R}}(r_{12}, r_{13}, r_{23})$ in the unit disc and a uniform distribution of nodes.  They then used this to place an upper bound on the RGG entropy, that is tighter than the previous best upper bound which assumes edge independence. Explicitly, if $\graph_n$ is the ensemble of (S)RGGs on $n$ nodes,
\begin{equation}
    \frac{H(\graph_n)}{\binom{n}{2}} \leq \frac{H(\graph_3)}{3} \leq H(\graph_2) = h_2(\bar{p}) 
\end{equation}
By performing these calculations in 1D, we are able to bound the RGG entropy of arbitrarily large RGGs in 1D using their method. 

For higher dimensions, and SRGGs (where the connection function introduces extra complexity in the calculation) we use a Monte-Carlo simulation to estimate the entropy. The details of the calculations will be given in the Appendix, and we show the main ideas here. \newline

\subsection{Exact Entropy Calculations}

In the following, we assume a \textit{hard} RGG formalism, with $d=1$, $n=3$, and $\nu$ being the uniform density on $\Omega$, which we will take as $\mathbb{T}$ and $[0,1]$ in turn. The distance metric will be $\rho_T$ (as defined in equation (\ref{eq:torus_metric})) for $\mathbb{T}$, and the Euclidean metric for $[0,1]$. \\

The strategy for computing the entropy is the same in each case: Let $p_k$ be the probability of a labelled graph drawn from a ensemble of 3-node graphs $\mathcal{G}$ being a specific labelled graph with $k$ edges. Then for a 3 node graph $p_0+3p_1+3p_2+p_3 = 1$, since there are 1,3,3 and 1 isomorphic graphs of edge count 0,1,2 and 3 respectively. We will calculate each of these probabilities individually.

\subsubsection{Exact Entropy Calculation for $\Omega = \mathbb{T}$}

We will start with $p_3$. Let $z_i$ be the position of the $i^{\text{th}}$ node. This probability is given by
\begin{equation}
    p_3 = \int_{[0,1]^3} \mathbb{I}\{\rho_T(z_1,z_2)<r_0\}\mathbb{I}\{\rho_T(z_1,z_3)<r_0\}\mathbb{I}\{\rho_T(z_2,z_3)<r_0\}dz_1dz_2dz_3
\end{equation}
 We first deal with the case when $r_0 < 1/3$ so that $p_0 > 0$. Now due to the translation invariance property, we may, without loss of generality, assume that $z_1 = 0$. Following the calculations through for $r_0 < 1/3$, then $1/3 < r_0 < 1/2$ (see Appendix \ref{exact_calcs}), we obtain:
\begin{table}[ht]
    \centering
    {\renewcommand{\arraystretch}{1.5}
    \begin{tabular}{ccc} \hline
        $p_k$ & $0< r_0 < 1/3$ & $1/3< r_0 < 1/2$ \\ \hline
        $p_3$ & $3r_0^2$ & $12r_0^2-6r_0+1$ \\
        $p_2$ & $r_0^2$ & $-8r_0^2+6r_0-1$ \\
        $p_1$ & $-5r_0^2+2r_0$ & $(2r_0-1)^2$ \\
        $p_0$ & $(1-3r_0)^2$ & $0$ \\ \hline
    \end{tabular}}
    \caption{Individual graph probabilities for the 3 node hard RGG in $\mathbb{T}$}
    \label{tab:torus_exact}
\end{table}\\
So let
\begin{gather}
H_{\mathbb{T}}^<(r_0) = -(1-3r_0)^2\log((1-3r_0)^2) - 3(2r_0-5r_0^2)\log(2r_0-5r_0^2) \nonumber \\- 3r_0^2\log(r_0^2) - 3r_0^2\log(3r_0^2)
\end{gather}
and
\begin{gather}
H_{\mathbb{T}}^>(r_0) = - (12r_0^2-6r_0+1)\log(12r_0^2-6r_0+1) \nonumber \\
- 3(6r_0 - 1 - 8r_0^2)\log(6r_0 - 1 - 8r_0^2) - 3(2r_0-1)^2\log((2r_0-1)^2)
\end{gather}
then
\begin{equation}
    H(\mathcal{G}) = \begin{cases}
    H_{\mathbb{T}}^<(r_0) & 0 < r_0 < 1/3 \\
    H_{\mathbb{T}}^>(r_0) & 1/3 \leq r_0 \leq 1/2 \\
    0 & \text{otherwise}
   \end{cases}
\end{equation}
By differentiating, we can show that this is maximised when $r_0 = 1/4$, thus for the 1D torus and 3 nodes
\[
\bar{p}_{max} = \int_0^{\frac{1}{2}} 2\mathbb{I}\left(r < \frac{1}{4}\right) dr = \frac{1}{2}
\]
At this point, we may wonder if the entropy is always maximised when the average connection probability is one half. The next example demonstrates that this is not the case.

\subsubsection{Exact Entropy Calculation for $\Omega = [0,1]$}
We will use the same method as before, but this time the integrals, whilst still tractable, become more intricate due to the removal of the translation invariance property. In order to simplify the problem, we will assume that $z_1 < z_2 < z_3$, then multiply the result by 6. We split the calculation into two cases, first when $r_0<1/2$, and then when $1/2<r_0<1$. Once again, the calculations are done in Appendix \ref{exact_calcs}.
\begin{table}[ht]
    \centering
    {\renewcommand{\arraystretch}{1.5}
    \begin{tabular}{ccc} \hline
        $p_k$ & $0< r_0 < 1/2$ & $1/2< r_0 < 1$ \\ \hline
        $p_3$ & $-2r_0^3+3r_0^2$ & $-2r_0^3+3r_0^2$ \\
        $p_2$ & $ - \frac{4}{3}r_0^3+r_0^2$ & $\frac{4}{3}r_0^3 - \frac{1}{3}({3}r_0-1)^2$ \\
        $p_1$ & $\frac{14}{3}r_0^3-6r_0^2+2r_0$ & $\frac{2}{3}(1-r_0)^{3}$ \\
        $p_0$ & $(1-2r_0)^{3}$ & $0$ \\ \hline
    \end{tabular}}
    \caption{Individual graph probabilities for the 3 node hard RGG in $[0,1]$}
    \label{tab:line_exact}
\end{table}

So, with
\begin{dmath}
    H_{[0,1]}^{<}(r_0) = -(3r_0^2-2r_0^3)\log(3r_0^2-2r_0^3) - 3\left(r_0^2 - \frac{4}{3}r_0^3\right)\log\left(r_0^2 - \frac{4}{3}r_0^3\right) -  3\left(\frac{14}{3}r_0^3 - 6r_0^2 + 2r_0\right)\log\left(\frac{14}{3}r_0^3 - 6r_0^2 + 2r_0\right) - (1-2r_0)^{3}\log(1-2r_0)^{3}
\end{dmath}
\begin{dmath}
    H_{[0,1]}^{>}(r_0) = -(3r_0^2-2r_0^3)\log(3r_0^2-2r_0^3) - (4r_0^3-(3r_0-1)^2)\log\left(\frac{4}{3}r_0^3-\frac{1}{3}(3r_0-1)^2\right) - (2(1-r_0)^{3})\log\left(\frac{2}{3}(1-r_0)^{3}\right)
\end{dmath}
we have
\begin{equation}
    H(\mathcal{G}) =
    \begin{cases}
        H_{[0,1]}^<(r_0) & 0 \leq r_0 < \frac{1}{2} \\
        H_{[0,1]}^>(r_0) & 1 \geq r_0 \geq \frac{1}{2} \\
        0 & \text{otherwise}
    \end{cases}
\end{equation}
Numerically evaluating the maximum of this function gives $\hat{r}_0 = 0.283$ and $\bar{p}_{max} = 0.486$. Notice that $\bar{p}_{max} \neq \frac{1}{2}$. The exact entropy curves for 3-node hard RGGs in $\mathbb{T}^1$ and $[0,1]$ are shown in Figure \ref{fig:exact}.

\begin{figure}
    \centering
    \includegraphics[width=0.85\linewidth]{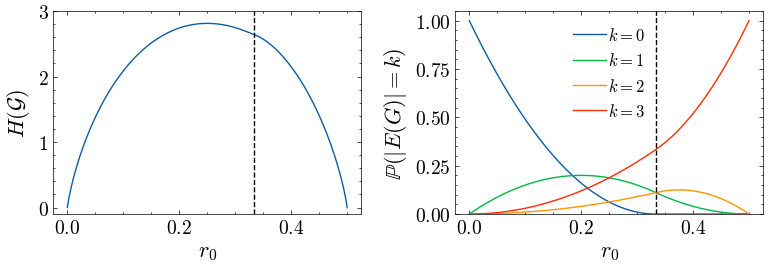}
    \includegraphics[width=0.85\linewidth]{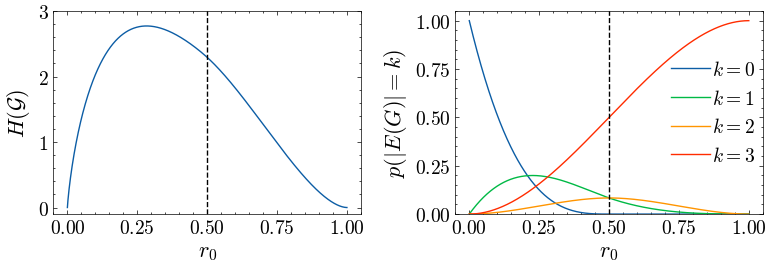}
    \caption{Exact entropy (left) and edge probability (right) curves in $r_0$ for 3-node hard RGGs with uniform node distributions in $\mathbb{T}^1$ (top) and $[0,1]$ (bottom). The dotted line shows the $r_0$ value where $p_0$ becomes 0, and where $p_1=p_2$.}
    \label{fig:exact}
\end{figure}

\subsection{Simulations for $d \in \{1,2,3\},$ $n=3$}
In this section we perform Monte-Carlo estimates of the entropy for 3 node hard and soft RGGs in $d=1,2$ and $3$. This will allow us to qualitatively compare the maximum entropy distributions for different types of RGG. For each combination of connection function and geometry, we generate $L=10^8$ graphs, and estimate the entropy of the ensemble by counting how many of each graph structure appear. Accurately estimating the entropy of SRGG ensembles is a particularly difficult task, due to the computational complexity. The value of $L$ required to accurately estimate the entropy must increase significantly as $n$ increases. If we enumerate the graphs in $\mathcal{G}$ as $\{g_1,g_2,...,g_{2^{\binom{n}{2}}}\}$, denote the estimated probability of each graph as $p_i$, and the associated entropy estimate as $\tilde{H}(\mathcal{G}) = -\sum_{i} p_i\log p_i$, then the relationship between the true entropy and the estimated entropy can be expressed as  \cite{roulston1999estimating} .
\begin{equation}
    \label{eq:error_decomposition}
    H(\mathcal{G}) = \tilde{H}(\mathcal{G}) + E_{total}
\end{equation}
\begin{equation}
    E_{total} \approx E_{sys} + \sigma\xi
\end{equation}
where $\xi$ is a standard Gaussian distributed random variable and 
\begin{equation}
    \label{e_sys}
    E_{sys} =  \frac{2^{\binom{n}{2}}-1}{2L}
\end{equation}
is the systematic error, and
\begin{equation}
    \label{eq:standard_error}
    \sigma = \sqrt{\frac{1}{L}\sum_{i=1}^{2^{\binom{n}{2}}}(\log(p_i)+\tilde{H}(\mathcal{G}))^2p_i(1-p_i)}
\end{equation}
is the standard error.  This estimate is derived from writing the error between the true and estimated probabilities as $\epsilon_i$, and taking a Taylor expansion to second order. Because the mean of $\epsilon_i$ is 0, the expected difference in the observed and true entropies is approximated by the second order term in the Taylor expansion, which is given by (\ref{e_sys}). The variance is calculated in a similar manner \cite{roulston1999estimating}. 

This means, to add another node and maintain the same degree of error, we would need to multiply $L$ by $\mathcal{O}(2^n)$, which becomes infeasible very quickly. In practice, Monte-Carlo entropy estimates are effective up to $n \approx 7$. Beyond this the estimates are slow and unreliable. So for now, we will continue with small numbers of nodes to demonstrate the effects of geometry and connection function choice. \newline

Table \ref{tab:max_entropy} shows the maximum entropy (entropy evaluated at $\hat{r}_0$) for different geometries and connection functions. The maximum possible entropy is $3\log(2) = 3$. In order to estimate the maximum entropy, we first make an approximation of where the maximum should be by simulating the entropy at various values of $r_0$, then estimate $\hat{r}_0$ by simulating the entropy at a range of $r_0$ values evenly spaced around the approximate maximum. We then approximate the entropy curve by fitting a quadratic using least squares regression to calculate the maximum. The resulting value of $\hat{r}_0$ is calculating as the maximum of the quadratic fit. By setting $\tilde{H}(\mathcal{G})$ to the maximum value of the quadratic approximation in equations (\ref{eq:standard_error}) and (\ref{eq:error_decomposition}), we obtain an estimate on the error, which is small enough to be confident in the third decimal place of the entropy estimates.  A more detailed explanation of this precision, and the method used to obtain an estimate on the numerical error can be found in Appendix \ref{error}. For each reported number in the table, the error is less than $10^{-3}$. 

\begin{table}[h]
    \centering
    \begin{tabular}{cccccccc}
        \hline
        Geometry & $\eta=1$ & $\eta=2$ & $\eta=3$ & $\eta=4$ & $\eta=5$ & $\eta=6$ & Hard  \\
        \hline
        $[0,1]$    & 2.994 & 2.945 & 2.888 & 2.849 & 2.826 & 2.812 & \textbf{2.771} \\
        $\mathbb{T}^1$       & 2.998 & 2.982 & 2.929 & 2.893 & 2.870 & 2.854 & \textbf{2.812}  \\
        $[0,1]^2$   & 2.999 & 2.992 & 2.978 & 2.960 & 2.949 & 2.940 & 2.907 \\
        $\mathbb{T}^2$ & 2.999 & 2.998 & 2.994 & 2.988 & 2.983 & 2.978 & 2.962 \\
        $[0,1]^3$     & 2.999 & 2.998 & 2.991 & 2.983 & 2.976 & 2.969 & 2.942 \\
        $\mathbb{T}^3$ & 2.999 & 2.999 & 2.999 & 2.996 & 2.994 & 2.991 & 2.975  \\
        \hline
    \end{tabular}
    \caption{Maximum entropy for $n=3$ in different geometries and connection functions with uniform node distributions. Values in bold are calculated analytically.}
    \label{tab:max_entropy}
\end{table}
In the low-dimensional regime, we see patterns emerging which hint at the behaviour we observe in the high-dimensional regime. The first trend to notice is that the maximum entropy decreases as $\eta$ increases. Intuitively, this is due to the connection function becoming `harder', so there is less uncertainty in the edge set given a fixed embedding. We also see that, for a fixed value of $\eta$, the entropy increases with dimension in both the cubic and toroidal geometries, however in general, the entropy estimate in the $d$-dimensional torus is greater than in the $d$-dimensional cube. In later sections, we will see that this is because the dependence between distances is smaller in the torus than in the cube, so the maximum entropy distribution can tend to a uniform distribution on the graphs faster. \\

 For the soft RGGs, it is interesting that the ensembles are nearly maximally complex even in this low-dimensional regime. The hard RGG column is worth particular mention. For example, the entropy in $[0,1]^2$ is greater than in $\mathbb{T}^1$, but the entropy in $[0,1]^3$ is less than in $\mathbb{T}^2$. We will see later (Theorems \ref{thm:torus_is_ER} and \ref{thm:hard_cube_entropy}) that this is a due to the fact that the nodes are uniformly distributed. In this setting, the hard RGG on the torus converges to the ER ensemble as $d$ increases, but the RGG on the cube converges to an ensemble with lower entropy.

\section{RGG Entropy in High Dimensions}
In this section we consider hard RGGs in $\Omega = [0,1]^d$ and $\Omega = \mathbb{T}^d$ with a fixed $n$, and take the limit $d\tti$. We will restrict ourselves to node distributions of the form
\begin{equation}
    \label{eq:general_distribution}
    \nu(\vec{x}) = \prod_{i=1}^d \pi(x_i)
\end{equation}
where $\pi$ is a probability measure on $[0,1]$. That is, the coordinates of the points in $\Omega$ are i.i.d.  This is a standard assumption in the RGG literature, but we expect it is possible to extend some of our results to non-i.i.d. coordinate distributions. For example, if the coordinates are identically distributed but not independent, as long as a suitable central limit theorem applies. This is outside the scope of this work, but would certainly be an interesting future research direction. 

For clarity, Table \ref{tab:main_results} shows the main results of this section. 

\begin{table}[h]
    \centering
    \begin{tabular}{cccc} \hline
        Geometry & Connection Function & $G(n,p)$ as $d\tti$? & $H(\graph)$ as $d\tti$ \\ \hline
        $\mathbb{T}^d$ - uniform nodes & Hard & \checkmark (Thm. \ref{thm:torus_entropy}) & $=H(G(n,p))$ \\ 
        $\mathbb{T}^d$ - non-uniform nodes & Hard & $\times$ (Thm \ref{no convergence for gaussian in torus thm}) & $<H(G(n,p))$ \\
        $\mathbb{T}^d$ & Soft & \checkmark (Thm. \ref{thm:soft_torus_entropy}) & $=H(G(n,p))$ \\ 
        $[0,1]^d$ & Hard & $\times$ (Thm. \ref{thm:hard_cube_entropy}) & $ < H(G(n,p))$ \\ 
        $[0,1]^d$ & Soft & \checkmark (Sec. IV of \cite{erba2020random}) & $=H(G(n,p))$ \\ \hline
    \end{tabular}
    \caption{Main results of Section 4. For each geometry and connection function, does the ensemble converge to the Erd\H os-R\'enyi ensemble, and is it possible to achieve the maximum entropy distribution?}
    \label{tab:main_results}
\end{table}

Clearly, if the ensemble converges to $G(n,p)$ as $d\tti$, then the entropy will converge to the entropy of $G(n,p)$ as well, hence it follows from the result of \cite{erba2020random} that the entropy of soft RGGs in $[0,1]^d$ tends to the entropy of $G(n,p)$ as $d\tti$, so we do not prove that explicitly here. In order for the distribution to converge to $G(n,1/2)$ as $d\tti$, we require that the connection function $p(r/r_0)$ must be able to converge to the value of $\frac{1}{2}$ as $d\tti$. That is, there exists a function $r_0(d)$ such that $\lim_{d\tti} p(r/r_0) = \frac{1}{2}$.  For continuous connection functions this is the case, as setting $r_0$ to a constant times $\sqrt{d}$ allows us to tune the connection probability parameter as we wish (see Theorem \ref{thm:soft_torus_entropy}). \\

For the hard RGG in $\mathbb{T}^d$ with uniform nodes, we can tune the parameters to achieve the maximum entropy distribution $G(n,1/2)$, but this is not the case for any other density of the form (\ref{eq:general_distribution}) in $\mathbb{T}^d$. We note that in \cite{erba2020random} it is shown that the hard RGG in $[0,1]^d$ does not converge to the ER ensemble for uniformly distributed nodes, but here we add to the result by showing that this occurs for \textit{any} distribution of nodes of the form (\ref{eq:general_distribution}), provided that the kurtosis of the variables in each coordinate is greater than 1. This is the case for any distribution, except the Bernoulli distribution with parameter $1/2$, or when the coordinates are almost surely constant (Appendix B). The reason that the hard and soft RGGs in the cube behave differently in the high-dimensional limit is that the soft connection function removes the dependence on the distance as $d\tti$, but the deterministic nature of the hard connection function alters the domain of integration which preserves the dependence.

\subsection{Central Limit Theorem for High Dimensional Cubes}
In this section, we will present central limit theorems for distances in high-dimensional cubes.  In reference \cite{erba2020random}, it was shown that soft RGGs on the cube converge to the ER ensemble. Also, the hard RGG in $[0,1]^d$ with nodes distributed with density of the form (\ref{eq:general_distribution}) \textit{does not necessarily} converge to the ER ensemble. In this work, we provide the necessary and sufficient condition for such convergence to occur. For these results, we need the multivariate central limit theorem, which is a consequence of the standard central limit theorem and the Cram\'er-Wold device.
We will now set up the central limit theorem for high-dimensional distance in cubes.
\begin{definition}[Normalised Euclidean Distance]
    \label{def:cube_distances}
    Suppose two points $X_i = (X_i^1,...X_i^d)$ and $X_j = (X_j^1,...,X_j^d)$ are independent and identically distributed according to a probability measure of the form (\ref{eq:general_distribution}) in $[0,1]^d$. Define the coordinate-wise  normalised squared distance  as
    \begin{equation}
        q_{ij}^k := |X_i^k-X_j^k|^2 - \mu_c
    \end{equation}
    where $\mu_c = \int_{{[0,1]^2}} |x-y|^2 \pi(x)\pi(y) dxdy$ is the average squared distance between two points in $[0,1]$. The  normalised squared distance  between $X_i$ and $X_j$ is
    \begin{equation}
        q_{ij} := \frac{1}{\sqrt{d}}\sum_{k=1}^d q_{ij}^k
    \end{equation}
    We let $\boldsymbol{q}^k$ collect the coordinate-wise pair distances,
    \begin{equation}
        \boldsymbol{q}^k := (q_{ij}^k)_{1\leq i < j \leq n}
    \end{equation}
    where $n$ is the number of nodes in the graph. Finally, we define the normalised connection range $t$ for an ensemble with connection range $r_0$ in the high dimension limit as
    \begin{equation}
        t = \lim_{d\tti} \frac{r_0^2}{\sqrt{d}}-\mu_c\sqrt{d}
    \end{equation}
\end{definition}
Using this notation, we are able to present the main result from \cite{erba2020random}, where we consider the specific case of the Euclidean norm.

\begin{lemma}[\cite{erba2020random}]
    \label{thm:cube_clt}
    Let $X_1,...,X_n$ be i.i.d. points distributed according to $\nu$ of the form (\ref{eq:general_distribution}) in $[0,1]^d$. Then in the limit $d\tti$, their  normalised squared distances  satisfy
    \begin{equation}
        \boldsymbol{q} := \frac{1}{\sqrt{d}} \sum_{i=1}^d \boldsymbol{q}^k \rightarrow Z
    \end{equation}
    in distribution, where $Z \sim N(0_{\binom{n}{2}}, \Sigma_c)$ is a Gaussian random vector, and $\Sigma_c = \{\mathbb{E}[|q_{ij}-\mu_c||q_{kl}-\mu_c|]\}_{1\leq i<j\leq n, 1\leq k<l \leq n}$ is the covariance matrix of  normalised squared distances  in the cube.
\end{lemma}

\begin{proof}
    Because the coordinates of a point distributed in $[0,1]^d$ are i.i.d., it is clear that the coordinate-wise pair distances is are also i.i.d. They are vectors distributed in $\mathbb{R}^{\binom{n}{2}}$ with zero mean, and therefore by the multivariate central limit theorem, their sum converges to the multivariate Gaussian. 
\end{proof}

\begin{remark}
    When the nodes are distributed uniformly at random in $[0,1]^d$, $\Sigma_c$ is given by $\Sigma_c^{(i,j),(k,l)} = \frac{7}{180}$ if $(i,j) = (k,l)$, $\Sigma_c^{(i,j),(k,l)} = \frac{1}{180}$ if $(i,j)$ and $(k,l)$ have exactly one element in common, and is 0 elsewhere. For this setting, $\mu_c=\frac{1}{6}$. These values are computed in Section \ref{sec:uniform}.
\end{remark}

\subsection{Central Limit Theorem for High-Dimensional Tori}
Here we will follow the same structure as the previous section, but this time replacing the cube with the torus. First, we will reformulate Definition \ref{def:cube_distances}.
\begin{definition}(Normalised Torus Distance)
    \label{def:torus_distances}
    Suppose two points $X_i = (X_i^1,...X_i^d)$ and $X_j = (X_j^1,...,X_j^d)$ are independent and identically distributed according to a probability measure of the form (\ref{eq:general_distribution}) in $\mathbb{T}^d$. Define the coordinate-wise  normalised squared distance  as
    \begin{equation}
        \tau_{ij}^k := \rho_T(X_i^k,X_j^k)^2 - \mu_T
    \end{equation}
    where $\mu_T = \int_{[0,1]^2} \rho_T(x,y)^2 \pi(x)\pi(y)dxdy$ is the average squared distance between two points in $\mathbb{T}^1$. The  normalised squared distance  between $X_i$ and $X_j$ is
    \begin{equation}
        \tau_{ij} := \frac{1}{\sqrt{d}}\sum_{k=1}^d \tau_{ij}^k
    \end{equation}
    We let $\boldsymbol{\tau}^k$ collect the coordinate-wise pair distances,
    \begin{equation}
        \boldsymbol{\tau}^k := (\tau_{ij}^k)_{1\leq i < j \leq n}
    \end{equation}
    where $n$ is the number of nodes in the graph. Finally, we define the normalised connection range $t$ for an ensemble with connection range $r_0$ in the high dimension limit as
    \begin{equation}
        t = \lim_{d\tti} \frac{r_0^2}{\sqrt{d}}-\mu_T\sqrt{d}
    \end{equation}
\end{definition}
\begin{lemma}
    \label{thm:torus_clt}
    Let $X_1,...,X_n$ be i.i.d. points distributed according to $\nu$ of the form (\ref{eq:general_distribution}) in $\mathbb{T}^d$. Then in the limit $d\tti$, their  normalised squared distances  satisfy
    \begin{equation}
        \boldsymbol{\tau} := \frac{1}{\sqrt{d}} \sum_{i=1}^d \boldsymbol{\tau}^k \rightarrow Z
    \end{equation}
    in distribution, where $Z \sim N(0_{\binom{n}{2}}, \Sigma_T)$ is a Gaussian random vector and $\Sigma_T = \{\mathbb{E}[(\tau_{ij}-\mu_T)(\tau_{kl}-\mu_T)]\}_{1\leq i<j\leq n, 1\leq k<l \leq n}$ is the covariance matrix of  normalised squared distances  in the torus.
\end{lemma}
\begin{proof}
    The proof is practically identical to that of Theorem \ref{thm:cube_clt}. By the same reasoning as in that proof, the vectors $\{\boldsymbol{\tau}^k\}_{k=1}^d$ satisfy the conditions of the multivariate CLT, and therefore tend to a multivariate Gaussian.
\end{proof}

\begin{remark}
    For uniform distributed points in $\mathbb{T}^d$, $\Sigma_T$ is given by $\Sigma_T^{(i,j),(k,l)} = \frac{1}{180}$ if $(i,j) = (k,l)$, and is 0 everywhere else. That is, the limiting distribution is the product distribution of $\binom{n}{2}$ independent normal random variables with mean 0 and variance $\frac{1}{180}$. In this setting, $\mu_T = \frac{1}{12}$. Again, see Section \ref{sec:uniform} for further details.
\end{remark}

\subsection{Entropy of Hard and Soft RGG Ensembles in High Dimensional Tori}
\subsubsection{Entropy of Hard RGG Ensembles in High Dimensional Tori}

 In this section, we will first show that if the nodes are independently and \textit{uniformly} distributed in each coordinate, then the ensemble of hard RGGs in the torus converges in distribution to the ER ensemble, which allows us to easily compute the entropy. We note that this is a known result, for example in \cite{friedrich2023cliques, baguley2025testing}, the same result is established for $n$ growing to infinity with $d$, but here we \textit{fix} $n$, and take $d\tti$. The proof of this result also serves to introduce notation and motivate the later results.
\begin{theorem}
    \label{thm:torus_entropy}
    Let $\mathcal{G}$ be an ensemble of $n$-node hard RGGs in $\mathbb{T}^d$, with independent, uniformly distributed nodes in each coordinate. Then as $d\tti$, $\graph \rightarrow G(n,\bar{p}(t))$ in distribution, where
    \begin{equation}
        \label{eq:gaussian_cdf}
        \bar{p}(t) = \prob{Z < t}
    \end{equation}
    where $Z \sim N(0, \alpha)$ is a Gaussian random variable with zero mean and variance $\alpha = \mathbb{E}[(\rho_T(X_i^1,X_j^1)^2-\mu_T)^2]$. 
\end{theorem}
\begin{proof}
    The proof involves calculating the elements of the covariance matrix $\Sigma_T$. There are three distinct values of the elements of $\Sigma_T$. The diagonal elements are given by
    \begin{equation}
        \alpha = \mathbb{E}[(\rho_T(X_i^1,X_j^1)^2-\mu_T)^2]
    \end{equation}
    where  $X_i, X_j$ are random points in $\mathbb{T}^d$ distributed according to $\nu$, and so $X_i^1$ and $X_j^1$ are independent random points distributed on $\mathbb{T}^1$ according to $\pi$.  The quantity $\alpha$ will be non-negative as it is the expectation of a square. The off diagonal elements are given by either
    \begin{equation}
        \beta = \Sigma_T^{(i,j),(j,k)} = \mathbb{E}[(\rho_T(X_i^1,X_j^1)^2-\mu_T)(\rho_T(X_j^1,X_k^1)^2-\mu_T)]
    \end{equation}
    or 
    \begin{equation}
        \gamma = \Sigma_T^{(i,j),(k,l)} = \mathbb{E}[(\rho_T(X_i^1,X_j^1)^2-\mu_T)(\rho_T(X_k^1, X_l^1)^2-\mu_T)] = 0
    \end{equation}
    by the definition of $\mu_T$. So we focus on $\beta$,
    \begin{equation}
        \beta = \int_{[0,1]^3} \rho_T(x_i,x_j)^2\rho_T(x_j,x_k)^2 dx_idx_jdx_k - \mu_T^2
    \end{equation}
    \begin{equation}
        = \int_{[0,1]^2} \rho_T(x_i,0)^2\rho_T(0,x_k)^2 dx_idx_k - \mu_T^2
    \end{equation}
    \begin{equation}
        =\mu_T^2-\mu_T^2 = 0
    \end{equation}
    Where in the second step we have used the translation invariance property of the $d$-torus with uniformly distributed nodes. So, $\Sigma_T$ is a diagonal matrix. \newline

    Next, we note that as a consequence of the central limit theorem for torus distances, we have that the probability of a graph given by (\ref{eq:graph_probability}) converges to the integral of the Gaussian density that is the limiting density of the pair distances as $d\tti$. Specifically, let $g \in \mathcal{G}$, then
    \begin{equation}
        \label{eq:general_torus_prob}
        \prob{g} = \int_{[0,D]^{\binom{n}{2}}} f_{\vec{R}}(\vec{r})\prod_{1\leq i < j \leq n} p(r_{ij}/r_0)^{x_{ij}}(1-p(r_{ij}/r_0))^{1-x_{ij}} d\vec{r}
    \end{equation}
    
    \begin{equation}
        = \prob{\bigcap_{(i,j) \in E(g)} \{R_{ij} < r_0\} \cap \bigcap_{(i,j) \notin E(g)} \{R_{ij} > r_0\}}
    \end{equation}
    \begin{equation}
        = \prob{\bigcap_{(i,j) \in E(g)} \{R_{ij}^2 < r_0^2\} \cap \bigcap_{(i,j) \notin E(g)} \{R_{ij}^2 > r_0^2\}}
    \end{equation}
    where $\vec{R} = (R_{12},...,R_{ij},...,R_{(n-1)n})$ is a vector of random distances, and $E(g)$ is the edge set of $g$. Now since each $R_{ij}^2$ is itself a sum of i.i.d. random variables, as $d\tti$, after rescaling the distances by taking away $\mu_Td$ and dividing by $\sqrt{d}$, the central limit theorem (lemma \ref{thm:torus_clt}) gives
    \begin{equation}
        \prob{g} \rightarrow \prob{\vec{Z} \in \mathcal{A}}
    \end{equation}
    \begin{equation}
        \label{eq:torus_joint_gaussian}
        = \int_{\mathcal{A}} N(0_{\binom{n}{2}},\Sigma_T)(\vec{z})d\vec{z}
    \end{equation}
    where $\mathcal{A}$ is given by
    \begin{equation}
        \mathcal{A} := \bigotimes_{1 \leq i < j \leq n} A_{ij}
    \end{equation}
    with $\bigotimes$ being the Cartesian product of sets, and
    \begin{equation}
        A_{ij} := \begin{cases}
            (-\infty, t], & x_{ij}=1 \\
            (t, \infty), & x_{ij}=0
        \end{cases}
    \end{equation}
    defining the region of $\mathbb{R}^{\binom{n}{2}}$ where the edge exists, with $t= \lim_{d\tti} d^{-\frac{1}{2}}(r_0^2 - \mu_Td)$ is the normalised connection range (when it exists).  Now, since $\Sigma_T$ is diagonal, 
    \begin{equation}
        \int_{\mathcal{A}} N(0_{\binom{n}{2}},\Sigma_T)(\vec{z})d\vec{z} = \prod_{1\leq i<j \leq n} \int_{A_{ij}} N(0,\alpha)(z)dz
    \end{equation}
    \begin{equation}
        = \prod_{1\leq i < j \leq n} \bar{p}(t)^{x_{ij}}(1-\bar{p}(t))^{1-x_{ij}}
    \end{equation}
    \begin{equation}
        = \bar{p}(t)^k(1-\bar{p}(t))^{\binom{n}{2}-k}
    \end{equation} 
    where $k = \sum_{i<j} x_{ij}$ is the number of edges in $g$. This is exactly the expression for the probability of a specific graph with $k$ edges in the ER ensemble, and therefore $\mathcal{G} \rightarrow G(n,\bar{p}(t))$ in distribution as $d\tti$.
\end{proof}
\begin{corollary}
    \label{thm:torus_is_ER}
    Let $\mathcal{G}$ be an ensemble of hard RGGs in $\mathbb{T}^d$ with independent, uniformly distributed nodes in each coordinate. Then
    \begin{equation}
        \lim_{d\tti} H(\mathcal{G}) = \binom{n}{2}h_2(\bar{p}(t))
    \end{equation}
    with $\bar{p} = \prob{Z<t}$ is the CDF of $Z \sim N(0, \alpha)$ with $\alpha$ as the diagonal element of $\Sigma_T$, and $t$ being the normalised connection range. Further, the entropy of $\graph$ is maximised in $r_0$ when $t=0$.
\end{corollary}
\begin{proof}
    Finding the limit of the entropy is simple, as we just need to find the entropy of the ER ensemble $G(n,\bar{p})$, which is known to be $\binom{n}{2}h_2(\bar{p})$. It is also known that the entropy of the ER ensemble is maximised at $\bar{p}(t)=1/2$, and so we need $t=0$ due to the symmetry of the Gaussian density. 
\end{proof}
\begin{remark}
    If we choose any $r_0(d)$ such that $\lim_{d\tti} \frac{r_0^2}{\sqrt{d}}-\mu_T\sqrt{d} = 0$, that is, $t = 0$, then we will maximise the entropy.  This tells us that the entropy-maximising choice of $r_0$ should scale as the square root of the dimension (at the same rate as the maximum distance of the domain), and the leading order of the scaling should be the square root of the mean squared distance in the domain.
\end{remark} 
\subsubsection{Non-Convergence to $G(n,p)$ for Non-Uniform Node Distributions on $\mathbb{T}^d$}
\label{no convergence for gaussian in torus}
\begin{theorem}
\label{no convergence for gaussian in torus thm}
    Suppose the ensemble $\graph$ of $n$ node hard RGGs on the $\mathbb{T}^d$ has nodes distributed according to a distribution of the form (\ref{eq:general_distribution}). Then $\graph$ converges in distribution to $G(n,p)$ as $d\tti$ if and only if the distribution of nodes in each coordinate is uniform.
\end{theorem}
\begin{proof}
    We have convergence to the ER ensemble if and only if the diagonal entries of the covariance matrix $\beta = 0$. Recall that, if $X,Y,Z$ are i.i.d. according to $\pi$,
    \begin{equation}
        \beta = \mathbb{E}[(\rho_T(X,Y)^2-\mu_T)(\rho_T(X,Y)^2-\mu_T)]
    \end{equation}
    So, by writing the expectation as an integral, and expanding the brackets, we have convergence to ER if and only if
    \begin{equation}
            \int_0^1\int_0^1\int_0^1\rho_T(x,y)^2\rho_T(x,z)^2 d\pi(x)\pi(y)\pi(z)dxdydz = \mu_T^2
    \end{equation}
    \begin{equation}
        \iff \int_0^1\left(\int_0^1 \rho_T(x,y)^2\pi(y)dy\right)^2\pi(x)dx = \mu_T^2
    \end{equation}
    \begin{equation}
        \iff \mathbb{E}_X[\mathbb{E}_Y[\rho_T(X,Y)^2]^2] = \mathbb{E}_X[\mathbb{E}_Y[\rho_T(X,Y)^2]]^2
    \end{equation}
    which is equivalent to
    \begin{equation}
        Var_X(\mathbb{E}_Y[\rho_T(X,Y)^2]) = 0
    \end{equation}
    and so $\mathbb{E}_Y[\rho_T(X,Y)^2]$ is almost surely constant. Rewriting this as an integral,
    \begin{equation}
        \int_0^1 \rho_T(x,y)^2\pi(y)dy = C
    \end{equation}
    $\pi$ almost-everywhere. In fact, this constant must be $\mu_T$, since integrating this function (which is constant) over the distribution of $x$ must give $\mu_T$. To show that the unique distribution satisfying this property is uniform, we will calculate the Fourier series of both sides. For any constant $C$, we have the $\xi^{\text{th}}$ Fourier mode of $C$, $\hat{C}(\xi)$ is:
    \begin{equation}
        \hat{C}(\xi) = \int_0^1 Ce^{-2\pi i \xi x}dx = \frac{C}{2\pi i \xi} \left(1-e^{-2\pi i\xi}\right) = 0
    \end{equation}
    whenever $\xi$ is a non-zero integer, and $\hat{C}_0 = C$. For the left hand side, write
    \begin{equation}
        F(x) = \int_0^1\rho_T(x,y)^2\pi(y)dy = \int_0^1 \rho_T(x-y)^2\pi(y)dy = (\rho_T^2 \ast \pi)(x) 
    \end{equation}
    where in the second inequality we have used a slight abuse of notation to show that $\rho_T(x,y)$ is dependent only on the difference between $x$ and $y$, and $\ast$ denotes the convolution of two functions. Here, we treat the functions $\pi$ and $\rho_T$ as periodic on $\mathbb{R}$ so we can apply the convolution theorem, but restrict them to $[0,1]$ so that we are working on $\mathbb{T}$. We have:
    \begin{equation}
        \label{eq:fourier_product}
        \hat{F}(\xi) = \widehat{\rho_T^2}(\xi)\hat{\pi}(\xi)
    \end{equation}
    where for a function $f$ we denote by $\hat{f}(\xi)$ its $\xi^{\text{th}}$ Fourier mode. We can write the Fourier modes of $\rho_T^2$ by splitting it over the cases when the argument is greater than, or less than $\frac{1}{2}$:
    \begin{equation}
        \widehat{\rho_T^2}(\xi) = \int_0^1 \rho_T(x)^2e^{-2\pi i \xi x}dx = \int_0^{\frac{1}{2}} |x|^2e^{-2\pi i \xi x}dx + \int_{\frac{1}{2}}^1(1-|x|)^2e^{-2\pi i \xi  x}dx
    \end{equation}
    \begin{equation}
        = \int_{-\frac{1}{2}}^{\frac{1}{2}} x^2e^{-2 \pi i \xi x}dz
    \end{equation}
    \begin{equation}
        = \begin{cases}
            \frac{1}{12} & \xi = 0 \\
            \frac{(-1)^{\xi}}{2\pi^2\xi} & \xi \in \mathbb{Z}\setminus\{0\}
        \end{cases}
    \end{equation}
    which is non-zero for every $\xi  \in \mathbb{Z}\setminus{0}$. Therefore (\ref{eq:fourier_product}) implies
    \begin{equation}
        \frac{1}{12}\hat{\pi}(0) = \mu_T
    \end{equation}
    and
    \begin{equation}
        \frac{(-1)^{\xi }}{2\pi^2\xi ^2}  \hat{\pi}(\xi ) = 0
    \end{equation}
    for $\xi \neq 0$. Therefore, we must have
    \begin{equation}
        \hat{\pi}(\xi)  = \begin{cases}
            12\mu_T & \text{if }\xi =0, \\
            0 & \text{if }\xi \neq0
        \end{cases}
    \end{equation}
    and therefore $\pi$ must be constant on $[0,1]$. Calculating $\mu_T$ for the uniform distribution, we find $\mu_T = \frac{1}{12}$, and therefore $\hat{\pi}(\xi) = 1$ if $\xi = 0$, and is $0$ elsewhere, so $\pi(x) = 1$ on $[0,1]$.
\end{proof}

As an example, suppose that in each coordinate, the nodes are distributed according to a Gaussian distribution restricted to $[0,1]$, that is,
\begin{equation}
    \label{eq:truncated_gaussian}
    \pi(x_i) = \begin{cases}
        \frac{1}{Z}\exp\left((x_i-\frac{1}{2})^2\right) & x_i \in [0,1] \\
        0 & \text{otherwise}
    \end{cases}
\end{equation}
where $Z = \int_0^1 \exp((x_i-\frac{1}{2})^2) dx_i$ is a normalising constant. From here onwards, when we refer to distributing nodes on $[0,1]^d$ or $\mathbb{T}^d$ according to the Gaussian distribution, we are referring to this truncated Gaussian distributed distribution, then using the respective distance metric to form the RGG.  Then, in the notation of Theorem \ref{thm:torus_entropy}, we have by numerical integration:
\begin{equation}
    \alpha \approx 0.005549
\end{equation}
\begin{equation}
    \beta \approx 0.000013
\end{equation}
the error on these estimates is less than $1.5\times 10^{-8}$. Note that $\beta > 0$. Therefore, the ensemble of RGGs with Gaussian distributed nodes \textit{does not} converge to the ER ensemble as $d\tti$, as there is a correlation between adjacent edges.
\subsubsection{Entropy of Soft RGG Ensembles in High Dimensional Tori}
We now consider ensembles of soft RGGs, and show that these converge to the ER ensemble for any distribution of nodes of the form (\ref{eq:general_distribution}).
\begin{theorem}
    \label{thm:soft_torus_entropy}
    Let $\mathcal{G}$ be an ensemble of $n$-node soft RGGs in $\mathbb{T}^d$ with  continuous  connection function $p(r/r_0)$ and i.i.d. nodes distributed according to $\nu$ of the form (\ref{eq:general_distribution}). Then if $r_0 = r_0(d) = \sqrt{kd} + o(\sqrt{d})$ as $d\tti$ for some constant $k$, we have that $\graph \rightarrow G(n,p_k)$ in distribution as $d\tti$, where $p_k := p(\sqrt{\mu_T/k})$.
\end{theorem}

\begin{proof}
    We can write the probability of an individual graph $g \in \graph$ as
    \begin{equation}
        \prob{g} = \int_{[0, \frac{\sqrt{d}}{2}]^{\binom{n}{2}}} f_{\vec{R}}(\vec{r})\prod_{i<j}p(r_{ij}/r_0)^{x_{ij}}(1-p(r_{ij}/r_0))^{1-x_{ij}} dr_{ij}
    \end{equation}
    Now let
    \begin{equation}
        r_{ij} = \sqrt{\tau_{ij}\sqrt{d}+\mu_Td}
    \end{equation}
    and we will transform the probability into the $\tau_{ij}$ domain. For ease of notation, define
    \begin{equation}
        \tilde{f}_d(\vec{\tau}) = f_{\vec{R}}\left(\sqrt{\tau_{12}\sqrt{d}+\mu_Td}, ...,\sqrt{\tau_{ij}\sqrt{d}+\mu_Td},..., \sqrt{\tau_{(n-1)n}\sqrt{d}+\mu_Td}\right)
    \end{equation}
    Writing $|J_d(\vec{\tau})|$ as the determinant of the Jacobian matrix of the transformation, we get
    \begin{equation}
        \prob{g} = \int_{[-\mu_T\sqrt{d}, \left(\frac{1}{4}-\mu_T\right)\sqrt{d}]^{\binom{n}{2}}} \tilde{f}_{d}(\vec{\tau})|J_d(\vec{\tau})|\ell_d(\vec{\tau})d\vec{\tau}
    \end{equation}
    where
    \begin{equation}
        \ell_d(\vec{\tau}) := \prod_{i<j} p\left(\frac{\sqrt{\tau_{ij}\sqrt{d}+\mu_Td}}{r_0}\right)^{x_{ij}}\left(1-p\left(\frac{\sqrt{\tau_{ij}\sqrt{d}+\mu_Td}}{r_0}\right)\right)^{1-{x_{ij}}}
    \end{equation}
    Now since $p$ is continuous, 
    \begin{equation}
        \lim_{d\tti} p\left(\frac{\sqrt{\tau_{ij}\sqrt{d}+\mu_Td}}{(1+o(1))\sqrt{kd}}\right) = p\left(\lim_{d\tti}\frac{\sqrt{\tau_{ij}\sqrt{d}+\mu_Td}}{(1+o(1))\sqrt{kd}}\right) = p\left(\sqrt{\frac{\mu_T}{k}}\right) =: p_k
    \end{equation}
    for any \textit{fixed} $\tau_{ij}$. Note that as $d\tti$, not every point in the domain of integration satisfies the above limit. For example, the end points of the domain of integration that scale as $\sqrt{d}$ do not converge to this limit. So, for any $l_{\infty}$ ball $U := \{x \in \mathbb{R}^{\binom{n}{2}} : \|x\|_{\infty} < u\}$ where $u \in \mathbb{R}$ is an arbitrary finite constant, we have uniform convergence of $\ell_d(\vec{\tau})$ to the limit 
    \begin{equation}
        \lim_{d\tti} \ell_{d}(\vec{\tau}) = \ell := \prod_{i<j}p_k^{x_{ij}}(1-p_k)^{1-x_{ij}}
    \end{equation}
    That is, for every $\delta >0$ we can choose a $d$ large enough so that $|\ell_d(\vec{\tau}) - \ell| < \delta$ for every $\vec{\tau}\in U$. Now we have
    \begin{equation}
        \mathbb{P}(g) = \int_{[-\mu_T\sqrt{d}, \left(\frac{1}{4}-\mu_T\right)\sqrt{d}]^{\binom{n}{2}}} \tilde{f}_{d}(\vec{\tau})|J_d(\vec{\tau})| \ell_d(\vec{\tau})d\vec{\tau} = \mathbb{E}_{\vec{\tau}}[\ell_d(\vec{\tau})]
    \end{equation}
    We want to show that $|\mathbb{E}_{\vec{\tau}}[\ell_d(\vec{\tau})] - \ell| \ttz$. Indeed, by Jensen's inequality,
    \begin{align}
        |\mathbb{E}_{\vec{\tau}}[\ell_d(\vec{\tau})] - \ell| \leq \mathbb{E}[|\ell_d(\vec{\tau}) - \ell|]
    \end{align}
    So, we need to bound the expectation $\mathbb{E}[|\ell_d(\vec{\tau}) - \ell|]$. Take any $l_{\infty}$ ball $U$ as above. Then $\ell_d(\vec{\tau}) \rightarrow \ell$ uniformly on $U$. We can write the expectation as
    \begin{equation}
        \mathbb{E}_{\vec{\tau}}[|\ell_d(\vec{\tau})-\ell|]  = \mathbb{E}_{\vec{\tau}}[|\ell_d(\vec{\tau})-\ell| \mathbb{I}(\vec{\tau} \in U)] + \mathbb{E}_{\vec{\tau}}[|\ell_d(\vec{\tau})-\ell|\mathbb{I}(\vec{\tau} \notin U)]  
    \end{equation}
    By uniform convergence on $U$, we can take $d$ large enough that the first term is less than $\epsilon_1$ for any $\epsilon_1 > 0$. For the second term, note that first, $\ell_d(\vec{\tau})$ is uniformly bounded above by 1, so
    \begin{equation}
        \mathbb{E}_{\vec{\tau}}[|\ell_d(\vec{\tau})-\ell|\mathbb{I}(\vec{\tau} \notin U)]  \leq  2\prob{\vec{\tau} \notin U}
    \end{equation}
    Then, if $Z \sim N(0,\Sigma_T)$,
    \begin{equation}
        \prob{\vec{\tau} \notin U} = \prob{\vec{\tau} \in U^c} - \prob{Z \in U^c} + \prob{Z \in U^c} 
    \end{equation}
    \begin{equation}
        \leq |\prob{\vec{\tau} \in U^c} - \prob{Z \in U^c}| + \prob{Z \in U^c} 
    \end{equation}
    and again, by the CLT in Lemma \ref{thm:torus_clt} we may choose $d$ large enough so that $|\prob{\vec{\tau} \in U^c} - \prob{Z \in U^c}|< \epsilon_2$ for any $\epsilon_2 > 0$. Then
    \begin{equation}
     \prob{Z \in U^c} = \prob{\bigcup_{i<j} \{|Z_{ij}| > u\}} \leq n\prob{|Z_{12}| > U}
    \end{equation}
    \begin{equation}
        \leq n\sqrt{\frac{2}{\pi}}\frac{\sigma}{u}e^{-\frac{u^2}{2\sigma^2}}
    \end{equation}
    where in the first line we used the union bound, and in the second we used Mill's inequality. The quantity $\sigma^2 = Var(Z) = Var(\tau_{ij})$ was given in the statement of Theorem \ref{thm:torus_entropy}, and is defined as $\alpha$. Since $\alpha$ is constant in $d$, $n$ is finite, and $u$ was chosen arbitrarily, we can make this term arbitrarily small (say, less than $\epsilon_3$) by selecting a large enough $u$. Then, for large enough $d$,
    \begin{equation}
        |\mathbb{E}_{\vec{\tau}}[\ell_d(\vec{\tau})] - \ell| \leq \epsilon := \epsilon_1 + 2(\epsilon_2+\epsilon_3)
    \end{equation}
    and so $\left|\mathbb{P}(g) - \prod_{i<j} p_k^{x_{ij}}(1-p_k)^{1-x_{ij}}\right| \rightarrow 0$ as required. 
\end{proof}

\begin{remark}
    This does not work for the hard RGGs because the deterministic connection function changes the limits of integration, and thus we cannot take the connection function out of the integral.
\end{remark}
To achieve the limiting graph ensemble with maximum possible entropy, we solve $p\left(\sqrt{\frac{\mu_T}{k}}\right) = \frac{1}{2}$ to give $G(n,\frac{1}{2})$.
\subsection{Entropy of Hard RGG Ensembles in the High Dimensional Cube}
This section is structured as follows. First, we will show that, under the Euclidean distance metric, and nodes distributed according to a density of the form (\ref{eq:general_distribution}), the RGG ensemble in the high dimensional cube with coordinate distribution $\pi$ that has kurtosis greater than 1 has a lower entropy than the ER ensemble. We will then show that the entropy of any such ensemble has a stationary point at $\bar{p}(t)=\frac{1}{2}$, and conjecture that this stationary point is a global maximum. To provide evidence for this claim, we then consider two specific distributions of nodes, the uniform and Gaussian distribution.

\subsubsection{General Setting}
For the following, we assume that $\mathcal{G}$ is an ensemble of hard RGGs in $[0,1]^d$ with nodes distributed according to a probability measure $\nu$ of the form (\ref{eq:general_distribution}).
\begin{definition}
    For an ensemble $\graph$ as above with connection range $r_0$, let $H_t(\graph)$ be the limit of the entropy $H(\graph)$ as $d\tti$, where $t = \lim_{d\tti} \frac{r_0^2}{\sqrt{d}} - \mu_c\sqrt{d}$ (if it exists) is the normalised connection range.
\end{definition}
\begin{theorem}
    \label{thm:hard_cube_entropy}
    For an ensemble $\graph$ as above,
    \begin{equation}
         H_t(\graph) < H\left(G\left(n,\frac{1}{2}\right)\right) = \binom{n}{2}
    \end{equation}
    provided that the coordinate-wise probability measure $\pi$ has kurtosis greater than 1.
\end{theorem}
\begin{proof}
    We will show that any choice of $\nu$ induces correlations between the edges. Then, since correlations must reduce entropy the result will follow. \newline

    As in the proof of Theorem \ref{thm:torus_entropy}, we will compute the off-diagonal elements of the covariance matrix, $\Sigma_c^{(i,j),(k,l)}$. As with the case of the torus, when $(i,j)$ and $(k,l)$ do not share an element, $\Sigma_c^{(i,j),(k,l)}=0$. When $(i,j)$ and $(k,l)$ share exactly one element, we have that, letting $X,Y,Z$ be independently distributed according to $\pi$ on $[0,1]$,
    \begin{equation}
        \beta := \mathbb{E}[(|X-Y|^2-\mu_c)(|Y-Z|^2-\mu_c|)]
    \end{equation}
    \begin{equation}
        = \mathbb{E}[|X-Y|^2|Y-Z|^2] - \mu^2
    \end{equation}
    \begin{equation}
        = \mathbb{E}[(X^2-2XY+Y^2)(Y^2-2YZ+Z^2)]-\mu^2
    \end{equation}
    \begin{gather}
        = \mathbb{E}[X^2Y^2 - 2YZX^2 + X^2Z^2 - 2XY^3 + 4XY^2Z \nonumber \\
        - 2XYZ^2 + Y^4 - 2Y^3Z+Y^2Z^2] - \mu^2
    \end{gather}
    \begin{equation}
        = 3\mathbb{E}[X^2]^2 - 4\mathbb{E}[X^3]\mathbb{E}[X] - \mu^2
    \end{equation}
    Where we have used that $X,Y$ and $Z$ are i.i.d. After noting that $\mu^2 = \mathbb{E}[|X-Y|^2]^2 = 4\mathbb{E}[X^2]^2-8\mathbb{E}[X^2]\mathbb{E}[X]^2+4\mathbb{E}[X]^4$ and factorising, we get that
    \begin{equation}
        \beta = E[(X-E[X])^4] - \mathbb{E}[(X-\mathbb{E}[X])^2]^2
    \end{equation}
    That is, adjacent distances are uncorrelated ($\beta=0$) only when the fourth central moment of $X$ is equal to the squared second central moment (variance) of $X$. In other words, the kurtosis of $X$ is 1, which is only satisfied by the Bernoulli distribution with parameter $\frac{1}{2}$ (see Appendix \ref{sec:bernoulli}). We can also see using Jensen's inequality that $\beta$ is always non-negative. Therefore, adjacent distances are \textit{positively correlated}. As a consequence of this, if $X_{ij}$ and $X_{ik}$ are random variables representing the edge indicators of edges $(i,j)$ and $(i,k)$ in $g\in\graph$, we have that they are positively dependent but identically distributed Bernoulli random variables, so
    \begin{equation}
        \prob{X_{ij}=1, X_{ik}=1} > \prob{X_{ij}=1}\prob{X_{ik}=1}
    \end{equation}
    and therefore the probability of the complete graph on $n$ nodes is greater than the probability of each edge existing independently. Thus, the distribution of graphs in $\graph$ cannot be uniform, and therefore the maximum possible entropy of $\binom{n}{2}$ is never achieved.
\end{proof}
We now show that the entropy always has a stationary point at $\bar{p}=\frac{1}{2}$.
\begin{proposition}
    Let $\mathcal{G}$ be an ensemble of $n$-node hard RGGs in $[0,1]^d$. Then in the $d\rightarrow\infty$ limit, the entropy has a stationary point at $\bar{p}=\frac{1}{2}$
\end{proposition}
We will prove this using two lemmas, which rely on the symmetry of an ensemble. Let $\mathbb{P}_t(g) := P(G=g)$ where $G$ is sampled from a hard RGG ensemble with normalised connection range $t$. We say that the ensemble of graphs is \textit{symmetric} about $a\in\mathbb{R}$ if for every $x \in \mathbb{R}$ and every graph $g\in \mathcal{G}$, we have that $\mathbb{P}_{a+x}(g) = \mathbb{P}_{a-x}(g^c)$, where $g^c$ is the complement of $g$, constructed by forming an edge whenever there is an edge missing in $g$, and vice versa.   

\begin{lemma}
    \label{lem:first}
    Let $H_t(\mathcal{G})$ be the entropy of a graph ensemble $\mathcal{G}$ with normalised connection range $t$ which is symmetric about $a$. Then $t=a$ is a stationary point of the entropy.
\end{lemma}
\begin{proof}
    Define $\graph^> := \{g \in \graph : g \text{ has }>\binom{n}{2} \text{ edges}\}$, $\graph^< := \{g \in \graph : g \text{ has }<\binom{n}{2} \text{ edges}\}$ and $\graph^= := \{g \in \graph : g \text{ has }\binom{n}{2} \text{ edges}\}$. Then due to symmetry, $g \in\graph^> \iff g^c \in \graph^<$, and $g \in \graph^= \iff g^c \in \graph^=$. We note that the contribution to the derivative of the terms in $\graph^=$ is 0. To see this, partition $\graph^=$ into two disjoint subsets so that the complement of every $g_1 \in \graph_1^=$ is in $\graph_2^=$.
    \begin{gather}
        \frac{\partial}{\partial x}\left(\sum_{g\in\mathcal{G}^=} \mathbb{P}_{a+x}(g)\log \mathbb{P}_{a+x}(g)\right) = \sum_{g_1\in\mathcal{G}_1^=} \frac{\partial \mathbb{P}_{a+x}(g_1)}{\partial x}\left(1+\log \mathbb{P}_{a+x}(g_1)\right) \nonumber \\
        + \sum_{g_2\in\mathcal{G}_2^=} \frac{\partial \mathbb{P}_{a+x}(g_2)}{\partial x}\left(1+\log \mathbb{P}_{a+x}(g_2)\right)
    \end{gather}
    \begin{equation}  
        =  \sum_{g_1\in\mathcal{G}_1^=} \left(\frac{\partial \mathbb{P}_{a+x}(g_1)}{\partial x}\left(1+\log \mathbb{P}_{a+x}(g_1)\right) - \frac{\partial \mathbb{P}_{a+x}(g_1)}{\partial x}\left(1+\log \mathbb{P}_{a-x}(g_1^c)\right)\right) = 0
    \end{equation}
    because $\partial_x \mathbb{P}_{a+x}(g) = -\partial_x \mathbb{P}_{a-x}(g^c)$. Thus the derivative is given by:
    \begin{equation}
        \frac{\partial H_{a+x}(\mathcal{G})}{\partial x} = -\sum_{g\in\mathcal{G}} \frac{\partial }{\partial x}\left(\mathbb{P}_{a+x}(g)\log \mathbb{P}_{a+x}(g) \right)
    \end{equation}
    \begin{equation}
        -\sum_{g\in\mathcal{G}^>} \frac{\partial }{\partial x}\left(\mathbb{P}_{a+x}(g)\log \mathbb{P}_{a+x}(g) \right) -\sum_{g\in\mathcal{G}^<} \frac{\partial }{\partial x}\left(\mathbb{P}_{a+x}(g)\log \mathbb{P}_{a+x}(g) \right)
    \end{equation}
    \begin{equation}
        = -\sum_{g\in\mathcal{G}^>} \frac{\partial \mathbb{P}_{a+x}(g)}{\partial x}\left(1+\log \mathbb{P}_{a+x}(g) \right)-\sum_{g\in\mathcal{G}^<} \frac{\partial \mathbb{P}_{a+x}(g)}{\partial x}\left(1+\log \mathbb{P}_{a+x}(g) \right)
    \end{equation}
    \begin{equation}
        = -\sum_{g\in\mathcal{G}^>} \frac{\partial \mathbb{P}_{a+x}(g)}{\partial x}\left(\log(\mathbb{P}_{a+x}(g)) - \log(\mathbb{P}_{a-x}(g)\right)
    \end{equation}
    When $x=0$, $a+x = a-x$, so $\mathbb{P}_{a+x}(g) = \mathbb{P}_{a-x}(g)$, so the derivative is equal to 0. Therefore we have a stationary point at $x=0$, so the normalised connection range of the ensemble is $t = a+0 = a$.
\end{proof}
Next, we show that $\bar{p}=\frac{1}{2}$ at the symmetric point.
\begin{lemma}
    \label{lem:second}
    For any distribution of graphs that is symmetrical about $a$, $\bar{p}=\frac{1}{2}$ at $t=a$.
\end{lemma}
\begin{proof}
    As before, partition $\mathcal{G}^=$ into two disjoint subsets, $\mathcal{G}_1^=$ and $\mathcal{G}_2^=$ such that the complement of every graph $g \in \mathcal{G}_1^=$ is in $\mathcal{G}_2^=$ and vice versa. Let $x \in \mathbb{R}$, and $E(g)$ be the edge set of $g$, then we can express $\bar{p}$ as
    \begin{equation}
        \bar{p} = \binom{n}{2}^{-1}\sum_{g \in \mathcal{G}} |E(g)|\mathbb{P}_{a+x}(g)
    \end{equation}
    \begin{gather}
        = \binom{n}{2}^{-1}\left(\sum_{g \in \mathcal{G}^<}|E(g)|\mathbb{P}_{a+x}(g) + \sum_{g \in \mathcal{G}^<}|E(g^c)|\mathbb{P}_{a+x}(g^c) \right. \nonumber \\
        + \left. \sum_{g \in \mathcal{G}_1^=} \left(\frac{1}{2}\binom{n}{2} \mathbb{P}_{a+x}(g) + \frac{1}{2}\binom{n}{2} \mathbb{P}_{a+x}(g^c)\right)\right)
    \end{gather}
    \begin{gather}
        = \binom{n}{2}^{-1}\left(\sum_{g \in \mathcal{G}^>}(|E(g)|\mathbb{P}_{a+x}(g) + \left(\binom{n}{2}-|E(g)|\right)\mathbb{P}_{a+x}(g)) \right. \nonumber \\
        + \left. \sum_{g \in \mathcal{G}_1^=} \left(\frac{1}{2}\binom{n}{2} \mathbb{P}_{a+x}(g) + \frac{1}{2}\binom{n}{2} \mathbb{P}_{a+x}(g^c)\right) \right)
    \end{gather}
    \begin{equation}
        =  \binom{n}{2}^{-1}\left(\sum_{g \in \mathcal{G}^>}\binom{n}{2}\mathbb{P}_{a+x}(g) + \sum_{g\in\mathcal{G}^=}\binom{n}{2}\mathbb{P}_{a+x}(g)\right)
    \end{equation}
    \begin{equation}
        = \sum_{g\in\mathcal{G}^> \cup \mathcal{G}_1^=} \mathbb{P}_{a+x}(g)
    \end{equation}
    Where the penultimate equality follows because the probability of any graph in $\mathcal{G}_1^=$ is equal to the probability of its complement, and the last equality is due to the symmetry of the ensemble. When $x=0$, we have that 
    \begin{equation}
        \sum_{g\in\mathcal{G}^>\cup \mathcal{G}_1^=} \mathbb{P}_{a+x}(g) = \sum_{g\in\mathcal{G}^>\cup \mathcal{G}_1^=} \mathbb{P}_{a-x}(g) = \sum_{g\in\mathcal{G}^<\cup \mathcal{G}_2^=} \mathbb{P}_{a+x}(g)
    \end{equation}
    So
    \begin{equation}
    1 = \sum_{g \in \mathcal{G}} \mathbb{P}_{a+x}(g) = \sum_{g\in\mathcal{G}^>\cup \mathcal{G}_1^=} \mathbb{P}_{a+x}(g) + \sum_{g\in\mathcal{G}^<\cup \mathcal{G}_2^=} \mathbb{P}_{a+x}(g) = 2\bar{p}
    \end{equation}
    So $\bar{p}=\frac{1}{2}$
\end{proof}
Then we can prove the proposition.
\begin{proof}[Proof of Proposition 1]
    By using Lemma 1, and then Lemma 2, it suffices to show that the Gaussian limit of $\mathcal{G}$ is symmetric about $t=0$. For a graph $g$, with edge indicators $\{x_{ij}\}_{1\leq i<j \leq n}$, construct the set $\mathcal{A}_g$ as
    \begin{equation}
        \mathcal{A}_g = \bigotimes_{1 \leq i < j \leq n} (-\infty, 0]^{x_{ij}}\times[0, \infty)^{1-x_{ij}}
    \end{equation}
    where $\bigotimes$ denotes the Cartesian product of sets and we set the value of a set to the power 0 as the empty set. Then $\mathcal{A}_g$ defines the region of $\mathbb{R}^{\binom{n}{2}}$ such that if a Gaussian sample is in $\mathcal{A}_g$, the resulting graph is the graph that $\mathcal{A}_g$ was constructed from. Note that $\mathbb{R}^{\binom{n}{2}} \setminus \mathcal{A}_g$ is constructed from the complement of the graph that constructed $\mathcal{A}_g$. So
    \begin{equation}
    \label{eq:gaussian_expectation}
        P(G = g) = \int_{\mathcal{A}_g} N(\boldsymbol{0},\Sigma_c)(\boldsymbol{q})d\boldsymbol{q} 
    \end{equation}
    Now since the individual coordinates of the Gaussian are symmetrical about 0,
    \begin{equation}
        P(G = g^c) = \int_{\mathbb{R}^{\binom{n}{2}}\setminus\mathcal{A}_g} N(\boldsymbol{0},\Sigma_c)(\boldsymbol{q})d\boldsymbol{q} = \int_{\mathcal{A}_g} N(\boldsymbol{0},\Sigma_c)(\boldsymbol{q})d\boldsymbol{q} = P(G = g)
    \end{equation}
    and therefore we are done by a direct application of Lemmas \ref{lem:first} and \ref{lem:second}.
\end{proof}
Through numerical simulations (in the sections below), it seems that the stationary point at $\bar{p}=\frac{1}{2}$ is a global maximum. Intuitively, the maximum entropy configuration seems that it should occur at the symmetric point, as there is a reasonably large probability of each graph occurring. It also seems that, from the low dimension simulations, $\bar{p}_{max} \rightarrow \frac{1}{2}$ as $d$ gets larger. However, we are at present unable to prove it.
\begin{conjecture}
    The maximum of $H_t(\graph)$ occurs when $t=0$, and therefore $\bar{p}_{max} \rightarrow \frac{1}{2}$ as $d\tti$.
\end{conjecture}
\begin{remark}
    In \cite{erba2020random}, it is shown that soft RGGs in the cube tend to the same ER limit as we have shown for the torus. Therefore, it requires the same method as in the torus to calculate their entropy in the $d\tti$ limit.
\end{remark}
\subsubsection{Uniform Distribution}
\label{sec:uniform}
Here, we will set $\nu(\vec{x}) = 1$, the uniform distribution on $[0,1]^d$. We find that, in this setting, $\mu_c = \frac{1}{6}$, and $\Sigma_c$ has the form of a diagonal matrix with diagonal entries equal to $\frac{7}{180}$, with off-diagonal entries equal to 0, except for $\Sigma_c^{(i,j),(i,k)} = \frac{1}{180}$. Let $\graph_k$ be the set of graphs in $\graph$ with $k$ edges. Define the `unnormalised' and `normalised' probabilities of a graph having $k$ edges as $P(g \in \graph_k)$ and $\binom{\binom{n}{2}}{k}^{-1}P(g\in\graph_k)$ respectively.  The normalised probabilities correspond to probabilities of a specific graph on $k$ edges occurring (since there are $\binom{k}{2}$ such graphs), and the unnormalised probabilities correspond to \textit{any} graph on $k$ edges occurring.  Figure \ref{fig:cube_dists} shows the normalised and unnormalised probabilities of a graph having $k$ edges in the as $d\tti$ in $[0,1]^d$ with a uniform distribution when $\bar{p}=\frac{1}{2}$, along with the comparison of the binomial distribution of edge counts that we would see in the $d\tti$ limit for $\mathbb{T}^d$.

\begin{figure}
    \centering
    \includegraphics[width=.9\linewidth]{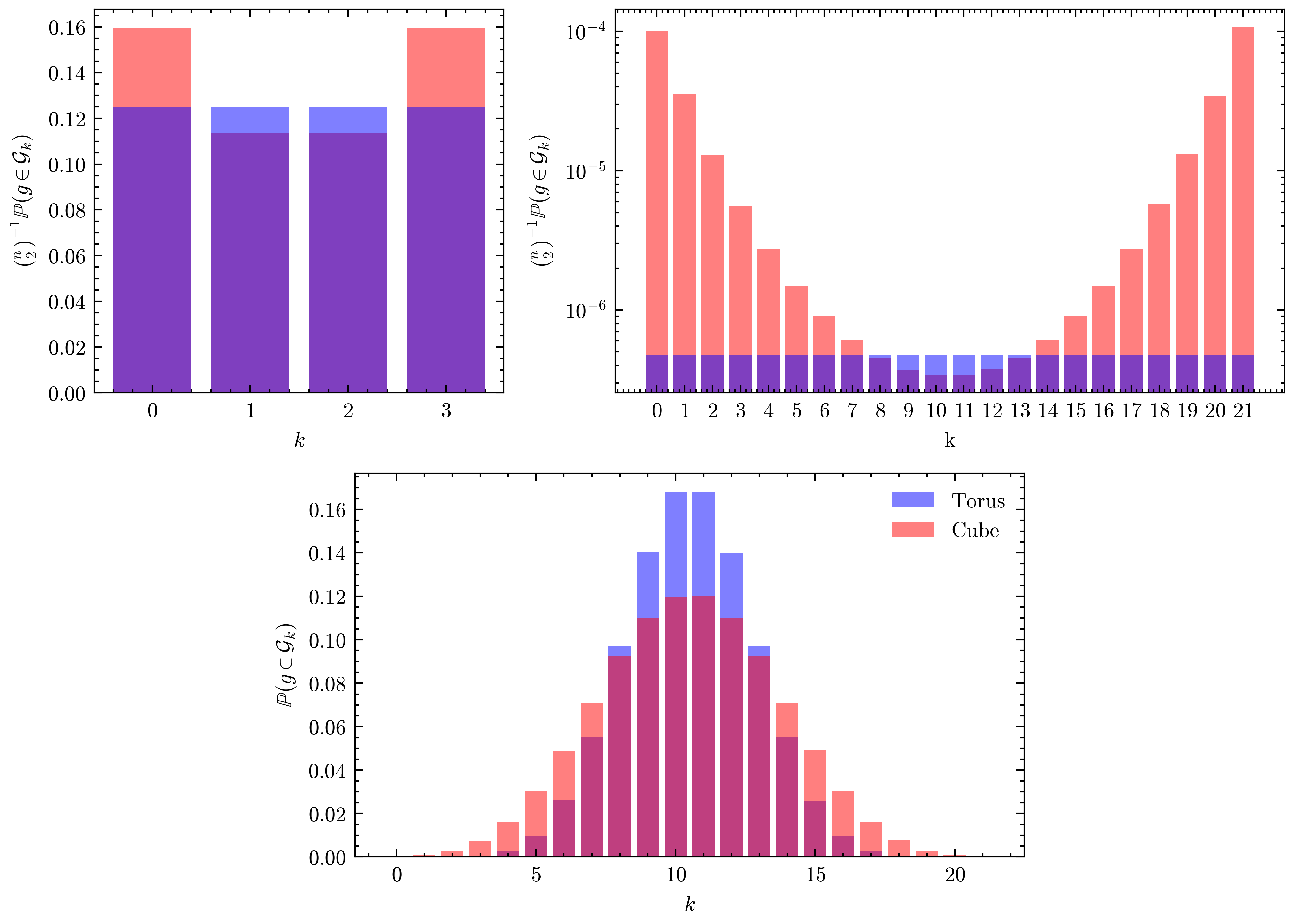}
    \caption{Top: Normalised distribution of edge count in the $d\tti$ limit for the hard RGG in the cube and torus, with uniformly distributed nodes for $n=3$ (top left), and $n=7$ (top right). Bottom: Comparison of the unnormalised distribution of edge counts for $n=7$ for hard RGGS with uniformly distributed nodes in the cube and torus.}
    \label{fig:cube_dists}
\end{figure}

\subsubsection{Truncated Gaussian Distribution}
Now, we compare the distribution of graphs in the cube and the torus for nodes distributed according to (\ref{eq:truncated_gaussian}). That is, we distribute points according to the truncated Gaussian, then use the respective distance metric to form the RGG. In the cube, we obtain $\mu_c\approx0.0243$ and $\Sigma_c$ has diagonal entries $\approx 0.0348$, and off-diagonal entries 0 except for $\Sigma_c^{(i,j),(i,k)}\approx 0.000434$. In the torus (as in Section \ref{no convergence for gaussian in torus}), we get $\mu_T \approx 0.006903$, and $\Sigma_T$ has diagonal elements $\approx 0.005549$, and off-diagonal entries 0 except for $\Sigma_T^{(i,j),(i,k)} \approx 0.000013$. We can see from Figure \ref{fig:gaussian_cube_dists} that the Gaussian distribution of nodes induces much heavier tails than the uniform distribution of Figure \ref{fig:cube_dists} in the cube, and we see a very slight, but noticeable deviation from the uniform distribution of graphs in the torus. Figure \ref{fig:unimodal} shows $H_t(\graph)$ against $t$ for $[0,1]^d$ in the $d\tti$ limit for $n=3$ and $11$ for both uniformly and Gaussian distributed nodes. We can clearly see strong evidence that the entropy has a unique maximum at $t=0$.

\begin{figure}
    \centering
    \includegraphics[width=.9\linewidth]{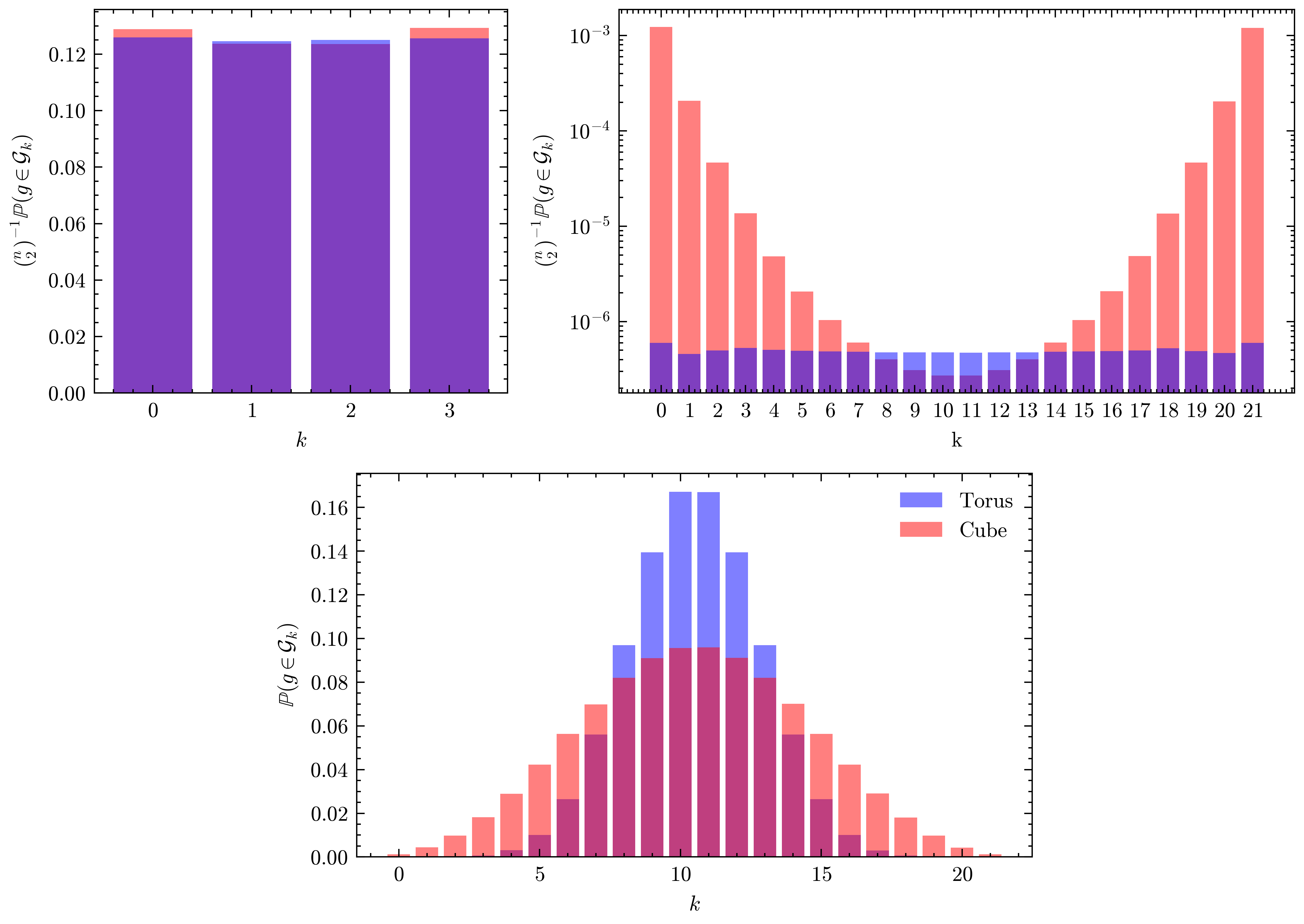}
    \caption{Top: Normalised distribution of edge count in the $d\tti$ limit for the hard RGG in the cube (red) and torus (blue), with Gaussian distributed nodes for $n=3$ (top left), and $n=7$ with a log scale for the probabilities (top right). Bottom: Comparison of the unnormalised distribution of edge counts for $n=7$ for hard RGGS with Gaussian distributed nodes in the cube and torus. }
    \label{fig:gaussian_cube_dists}
\end{figure}

\begin{figure}
    \centering
    \includegraphics[width=0.9\linewidth]{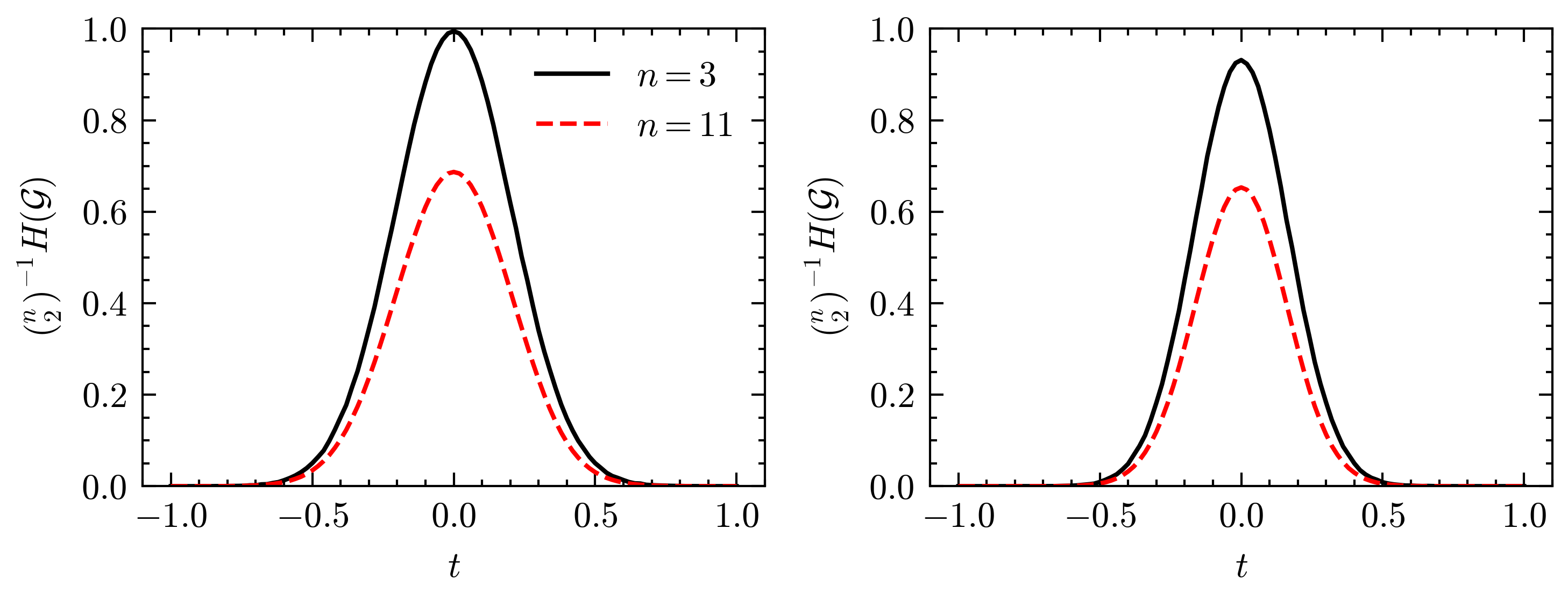}
    \caption{Entropy curves in the theoretical $d\tti$ limit for hard RGGs in the cube, with $n=3$ and $11$ nodes and uniformly (left) and Gaussian (right) distributed nodes. All 4 curves show evidence for a unique maximum at $t=0$.}
    \label{fig:unimodal}
\end{figure}
\section{Edgeworth Expansion for Estimation in Lower Dimensions}
Finally, we will develop a third-order correction to the central limit theorems for distances in $[0,1]^d$ and $\mathbb{T}^d$ in the specific case of the uniform distribution of nodes, as it makes the computation of the correction available in closed form. In theory one could perform this correction for any distribution of nodes of the form (\ref{eq:general_distribution}). \newline

The central limit theorem means we can approximate high dimensional graph entropy to a reasonable degree of accuracy. However, because the formula for the graph probability in the limit does not include the dimension, it does not give us any information about how the entropy scales in dimension. One might hope to do better by including higher order corrections to the central limit theorem. For multivariate CLTs, these come in the form of Edgeworth expansions \cite{withers2000simple}. For a multivariate density function $f$ that converges to a Gaussian density as $d\tti$, we can write the third-order Edgeworth expansion for a finite $d$ as
\begin{equation}
    \label{eq:edgeworth_expansion}
    f(\boldsymbol{q}) \approx N(\boldsymbol{0}, \Sigma)(\boldsymbol{q})\left(1 + \frac{1}{6\sqrt{d}}\sum_{|\alpha| = 3} \mathbb{E}[\boldsymbol{q}^{\alpha}]\mathcal{H}_{\alpha}(\boldsymbol{q})\right)
\end{equation}
Here, $\alpha$ is a vector that indicates which elements of the vector $\boldsymbol{q}$ we should include. For example, if we want to indicate that the expectation is over $q_{(1,2)}^2q_{(1,4)}$, we set $\alpha = (2,0,1,0,...,0) \in \mathbb{N}^{\binom{n}{2}}$. The notation $|\alpha| = 3$ indicates that we should sum over all third-order correlations. $\mathcal{H}_{\alpha}(\cdot)$ is the Hermite polynomial corresponding to $\alpha$, given by
\begin{equation}
    (-1)^{|\alpha|} \partial^{\alpha} N(\boldsymbol{0}, \Sigma)(\boldsymbol{q}) = \mathcal{H}_{\alpha}(\boldsymbol{q})N(\boldsymbol{0}, \Sigma)(\boldsymbol{q})
\end{equation}
where the $\alpha$ in the derivative denotes which derivatives we should take. For example $\alpha = (2,0,1,0,...0)$ means we take the second derivative of the first component, and the first derivative of the third component \cite{skovgaard1986multivariate}. This can be interpreted as including an estimate for the skew of the distribution, which is given by the third moments. To assess the improvement in the multivariate expansion, we will evaluate (\ref{eq:edgeworth_expansion}). By symmetry, there are only 8 instances of $\alpha$ we need to evaluate for any number of nodes $n \geq 3$, corresponding to the number of combinations of 
\begin{equation}
    \label{eq:big_integral}
    \mathbb{E}[q_{ij}q_{kl}q_{mn}] := \int_{[0,1]^6} (\rho(x_i,x_j)-\mu)(\rho(x_k,x_l)-\mu)(\rho(x_m,x_n)-\mu)dx_idx_jdx_kdx_ldx_mdx_n
\end{equation}
there are. Above we let $\rho(\cdot, \cdot)$ stand for either the distance metric on the hypercube or the hypertorus, and therefore $\mu$ will take either the value of $\mu_c$ or $\mu_T$ in each case.

\begin{table}[]
    \centering
    \begin{tabular}{m{5cm}|m{7cm}} \hline
        \hspace{1cm} Correlation Type & \hspace{2.5cm} Diagram \\ \hline
        \hspace{1cm}$e_1 = \mathbb{E}[q_{ij}^3]$ & \vspace{0.25cm} \hspace{2.5cm} \includegraphics[width=0.25\linewidth]{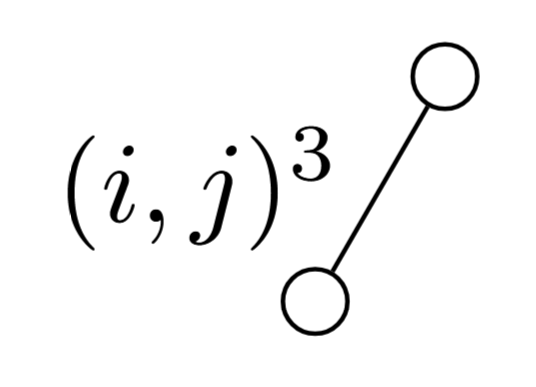} \vspace{0.25cm} \\ \hline
        \hspace{1cm}$e_2 = \mathbb{E}[q_{ij}^2q_{ik}]$ &  \vspace{0.25cm} \hspace{2.5cm} \includegraphics[width=0.3\linewidth]{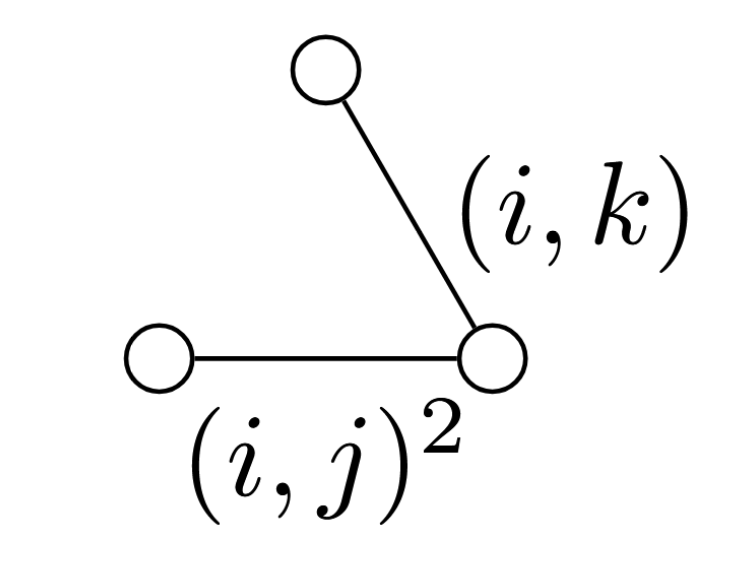}  \vspace{0.25cm}\\ \hline
        \hspace{1cm}$e_3 = \mathbb{E}[q_{ij}q_{jk}q_{ki}]$ & \vspace{0.25cm} \hspace{2.5cm} \includegraphics[width=0.32\linewidth]{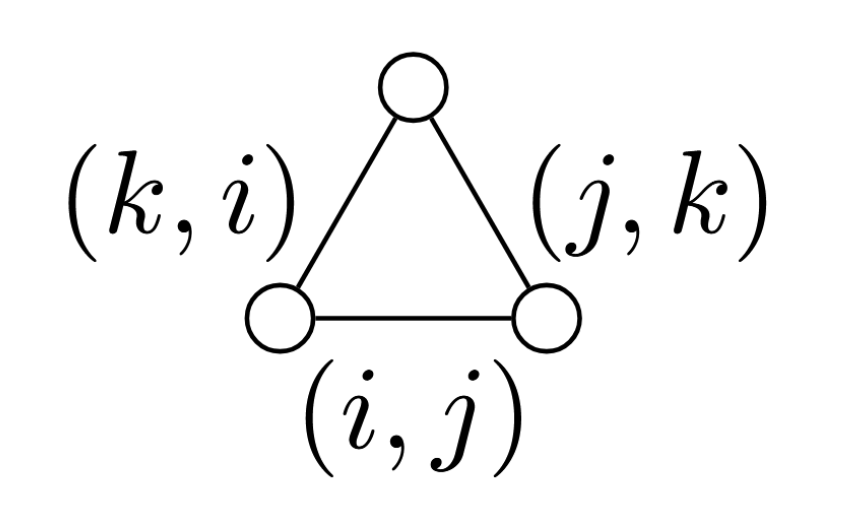} \vspace{0.25cm} \\ \hline
            \hspace{1cm}$e_4 = \mathbb{E}[q_{ij}q_{ik}q_{il}]$ &  \vspace{0.25cm} \hspace{2.3cm} \includegraphics[width=0.32\linewidth]{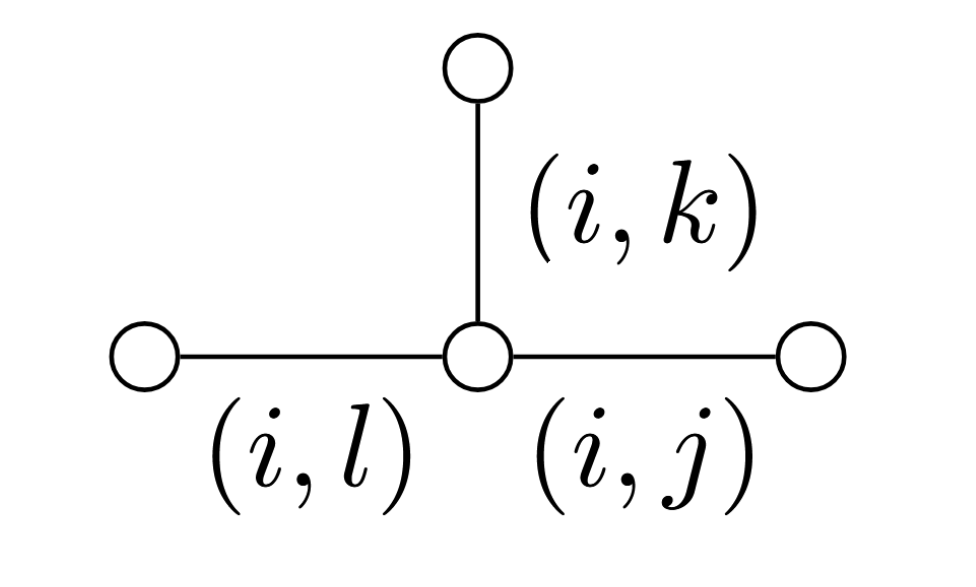} \vspace{0.25cm} \\ \hline
        \hspace{1cm}$e_5 = \mathbb{E}[q_{ij}q_{kl}q_{mn}]$ &  \vspace{0.25cm} \hspace{2cm} \includegraphics[width=0.45\linewidth]{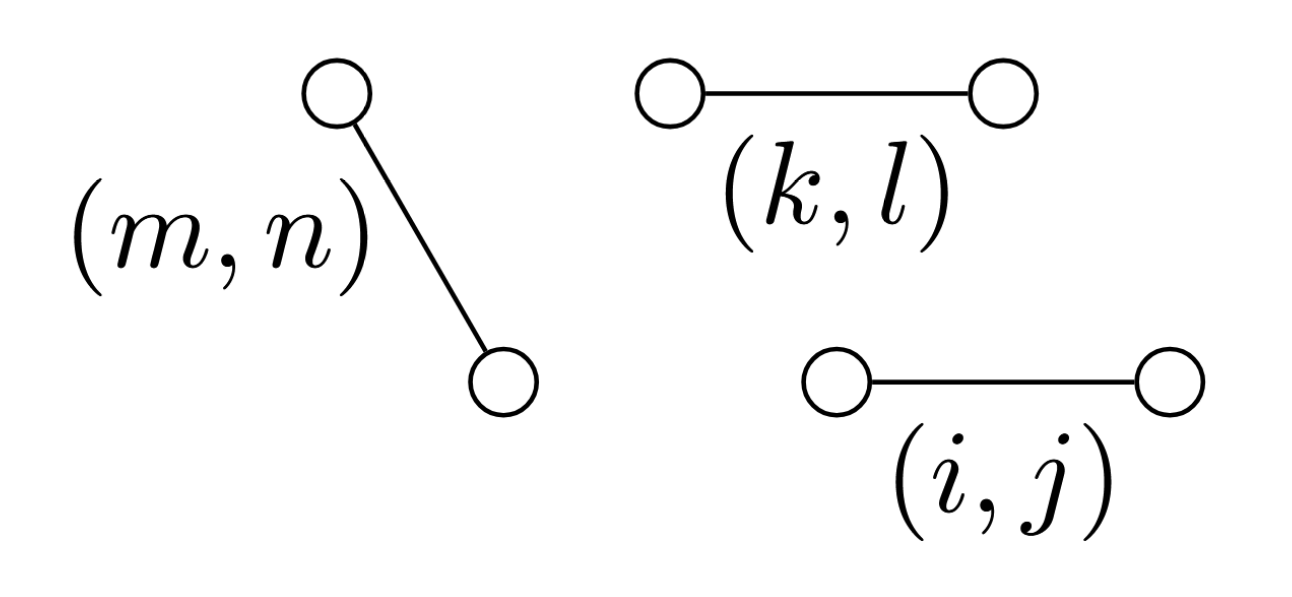} \vspace{0.25cm} \\ \hline
        \hspace{1cm}$e_6 = \mathbb{E}[q_{ij}^2q_{mn}]$ &  \vspace{0.25cm}\hspace{2.2cm} \includegraphics[width=0.4\linewidth]{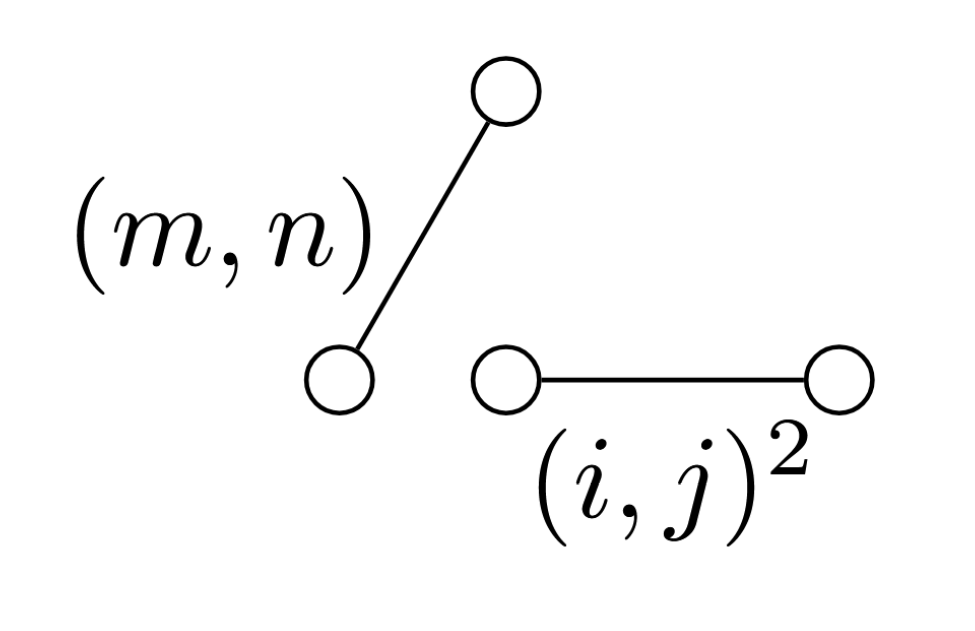} \vspace{0.25cm} \\ \hline
       \hspace{1cm}$e_7 = \mathbb{E}[q_{ij}q_{jk}q_{mn}]$ &  \vspace{0.25cm}\hspace{2.2cm} \includegraphics[width=0.45\linewidth]{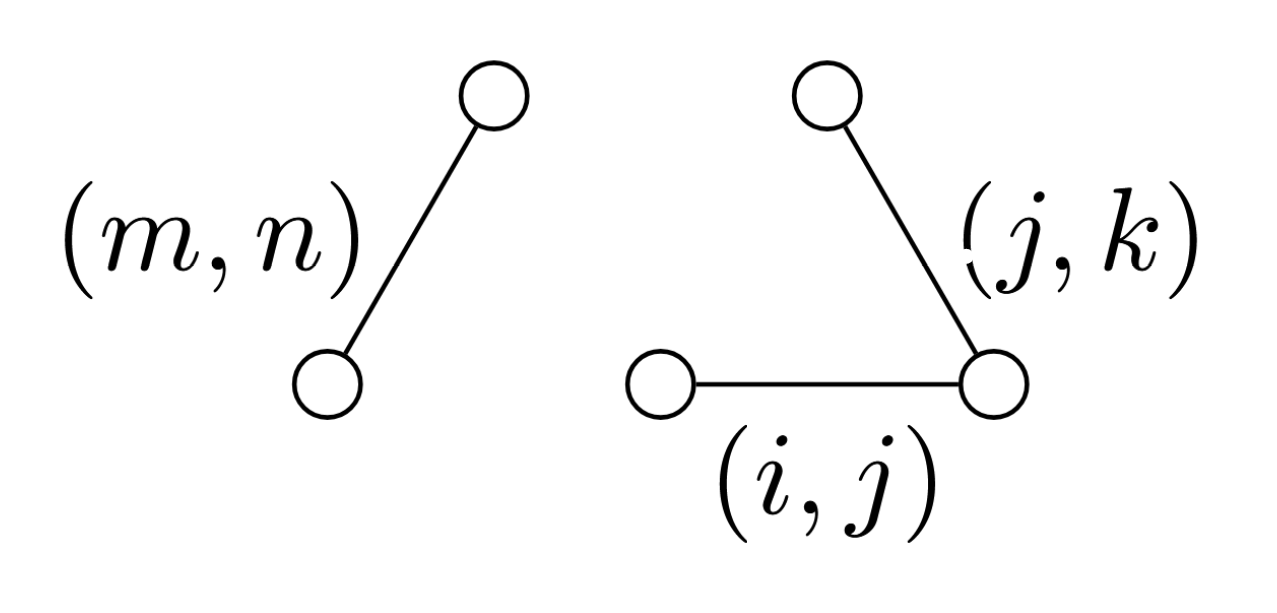} \vspace{0.25cm} \\ \hline
        \hspace{1cm}$e_8 = \mathbb{E}[q_{ij}q_{jk}q_{kl}]$ &  \vspace{0.25cm}\hspace{2.2cm} \includegraphics[width=0.4\linewidth]{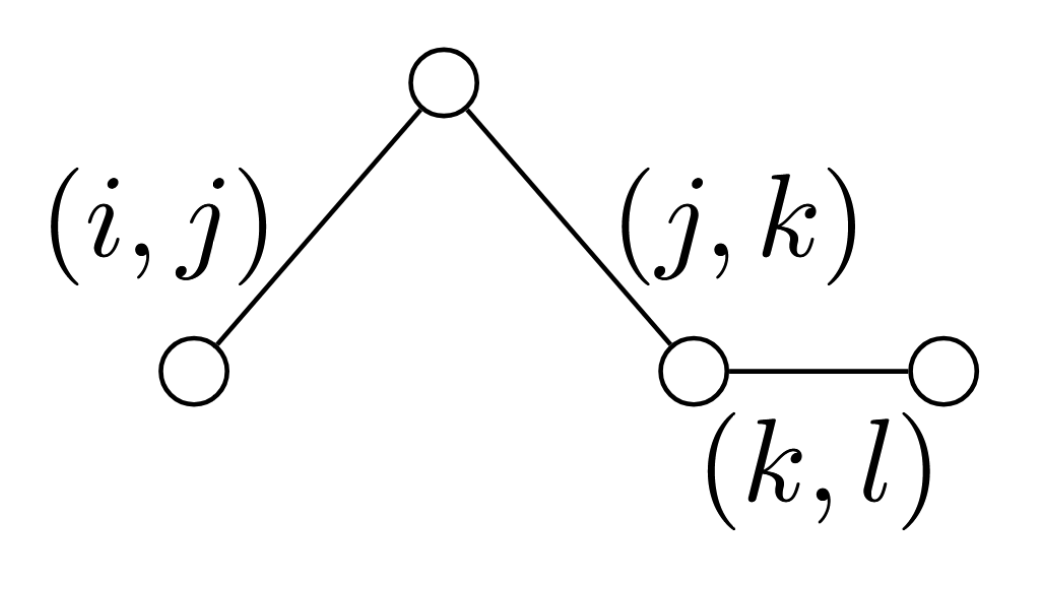}  \vspace{0.25cm}\\ \hline
    \end{tabular}
    \caption{Graphical Representation of the 8 Correlation Types}
    \label{tab:my_label}
\end{table}

 See Table \ref{tab:my_label} for a graphical representation of these correlations. When we compute these expectations for the cube, we find that only the first four are non-zero. These correspond to (substituting $q$ for $\tau$ if we are working in the hypertorus)
\[
    e_1 = \mathbb{E}[q_{ij}^3]
\]
\[
    e_2 = \mathbb{E}[q_{ij}^2q_{ik}]
\]
\[
    e_3 = \mathbb{E}[q_{ij}q_{jk}q_{ki}]
\]
\[
    e_4 = \mathbb{E}[q_{ij}q_{ik}q_{il}]
\]
$e_5,e_6$ and $e_7$ are clearly zero as they all contain one independent edge. Therefore the expectation factorises, and the fact themean of each $q_{ij}$ is 0 implies the whole expectation is 0. That $e_8 = 0$ is given by evaluating the integral (\ref{eq:big_integral}). Next, we need to evaluate the multivariate Hermite polynomials. Due to \cite{withers2000simple}, we have a simple way to do this. If we let $z = \Sigma^{-1}\boldsymbol{q}$, and $Z \sim N(\boldsymbol{0}, \Sigma^{-1})$ (where once again $\Sigma$ can refer to $\Sigma_c$ or $\Sigma_T$ depending on which geometry we are working in), then the Hermite polynomial corresponding to $\alpha$ is given by 
\begin{equation}
    \mathcal{H}_{\alpha}(\boldsymbol{q}) = \mathbb{E}[(z+Z\sqrt{-1})^{\alpha}]
\end{equation}
Writing $H_i$ to correspond to the configuration of distances that correspond to the expectation $e_i$, then we have (recalling that the covariance matrix is index by multi-indices)
\[
    H_1(q_{ij}) = z_{ij}^3 - 3z_{ij}\Sigma^{-1}_{(i,j),(i,j)} 
\]
\[
    H_2(q_{ij}, q_{ik}) = z_{ij}^2z_{ik} - z_{ik}\Sigma^{-1}_{(i,j),(i,k)} - z_{ij}\Sigma^{-1}_{(i,j),(i,j)}
\]
\[
    H_3(q_{ij}, q_{jk}, q_{ki}) = z_{ij}z_{jk}z_{ki} - \Sigma^{-1}_{(i,j),(i,k)}(q_{ij} + q_{jk} + q_{ki})
\]
\[
    H_4(q_{ij}, q_{ik}, q_{il}) = z_{ij}z_{ik}z_{il} - \Sigma^{-1}_{(i,j),(i,k)}(q_{ij} + q_{ik} + q_{il})
\]
Then by substituting this back into (\ref{eq:edgeworth_expansion}), we are able to estimate the probabilities of a given graph by computing a Gaussian expectation. With $\mathcal{A}_g$ as in (\ref{eq:gaussian_expectation}), we have
\begin{equation}
    \label{eq:edgeworth_correction}
    \mathbb{P}_t(g) \approx \mathbb{E}_{\boldsymbol{q}}\left[\mathbb{I}(\boldsymbol{q} \in \mathcal{A}_g)\left(1 + \frac{1}{6\sqrt{d}}\sum_{|\alpha| = 3} \mathbb{E}[\boldsymbol{q}^{\alpha}]\mathcal{H}_{\alpha}(\boldsymbol{q})\right)\right]
\end{equation}
Let $\mathbb{P}(g)$ be the probability of a graph given by the Gaussian limit. and $\mathbb{P}_d(g)$ be the probability of the same graph in a finite dimension $d$. Then the Edgeworth expansion (\ref{eq:edgeworth_correction}) gives that
\begin{equation}
    \mathbb{P}(g) = \mathbb{P}_d(g) + \bigO{d^{-\frac{1}{2}}}
\end{equation}
If we then denote $H(\graph)$ as the entropy given by the Gaussian limit, and $H_d(\graph)$ be the entropy given by the Edgeworth correction, then
\begin{equation}
    H(\graph) = -\sum_{g\in\graph} \prob{g}\log \prob{g}
\end{equation}
\begin{equation}
    = - \sum_{g\in \graph}(\mathbb{P}_d(g) + \bigO{d^{-\frac{1}{2}}})\log \left(\mathbb{P}_d(g) + \bigO{d^{-\frac{1}{2}}}\right))
\end{equation}
\begin{equation}
    = -\sum_{g \in \graph} \left[\mathbb{P}_d(g)(\log (\mathbb{P}_d(g)) + \log \left(1+\bigO{d^{-\frac{1}{2}}}\right) + \bigO{d^{-\frac{1}{2}}}\right]
\end{equation}
\begin{equation}
    = -\sum_{g\in\graph} \left[\mathbb{P}_d(g) \log(\mathbb{P}_d(g)) + \log\left(1+\bigO{d^{-\frac{1}{2}}}\right)\mathbb{P}_d(g) + \bigO{d^{-\frac{1}{2}}}\right]
\end{equation}
\begin{equation}
    \label{eq:entropy_scaling_in_d}
    = H_d(\graph) + \bigO{d^{-\frac{1}{2}}}
\end{equation}

Using the Edgeworth estimate for the individual graph probabilities, we can estimate the ensemble entropy. Figure \ref{fig:entropy_estimate} shows the entropy estimates given by the Gaussian, Edgeworth and simulation methods for entropy estimation. We fitted a curve of the form $H_d(\graph) = a - b(d^{-\frac{1}{2}}+c)$ to match the form of (\ref{eq:entropy_scaling_in_d}). The Edgeworth estimate is within 2 decimal places of the simulated entropy for $d \geq 15$. The advantage of using this method is that the computation time is independent of dimension because the third-order expectations are the same for each dimension, so it can evaluate an estimate to the entropy at any dimension with the same speed, which makes it much faster than the simulation method for higher dimensions. In practice, the simulation method was faster for $d \leq 75$ whereas the Edgeworth method became faster for $d > 75$. The increase in speed comes with a tradeoff of accuracy, and at a high enough dimension ($d\sim250$) it is much faster (and more accurate) to use the Gaussian method, with no Edgeworth corrections. In the torus, we see very similar patterns. In both geometries, the Edgeworth entropy estimate is consistently less than the simulated entropy estimate. This is due to a combination of numerical error, which detracts from the `uniformity' of the resulting distribution, which can only have the effect of reducing entropy. Also, the higher moments missing from the expansion will contribute to the error. In the torus the Edgeworth estimate runs faster than the simulation method in every dimension.

\begin{figure}
    \centering
    \includegraphics[width=0.48\linewidth]{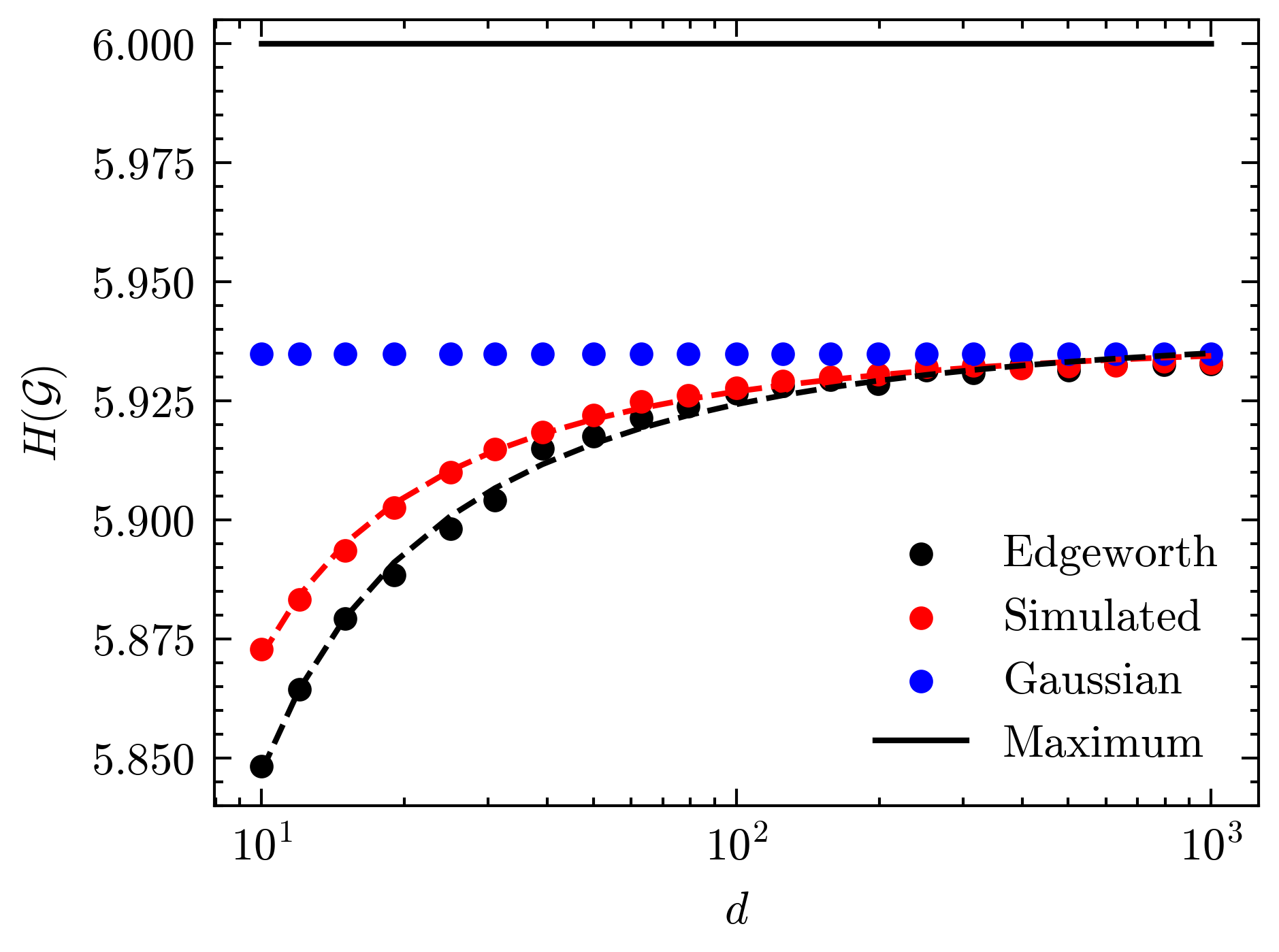}
    \includegraphics[width=0.48\linewidth]{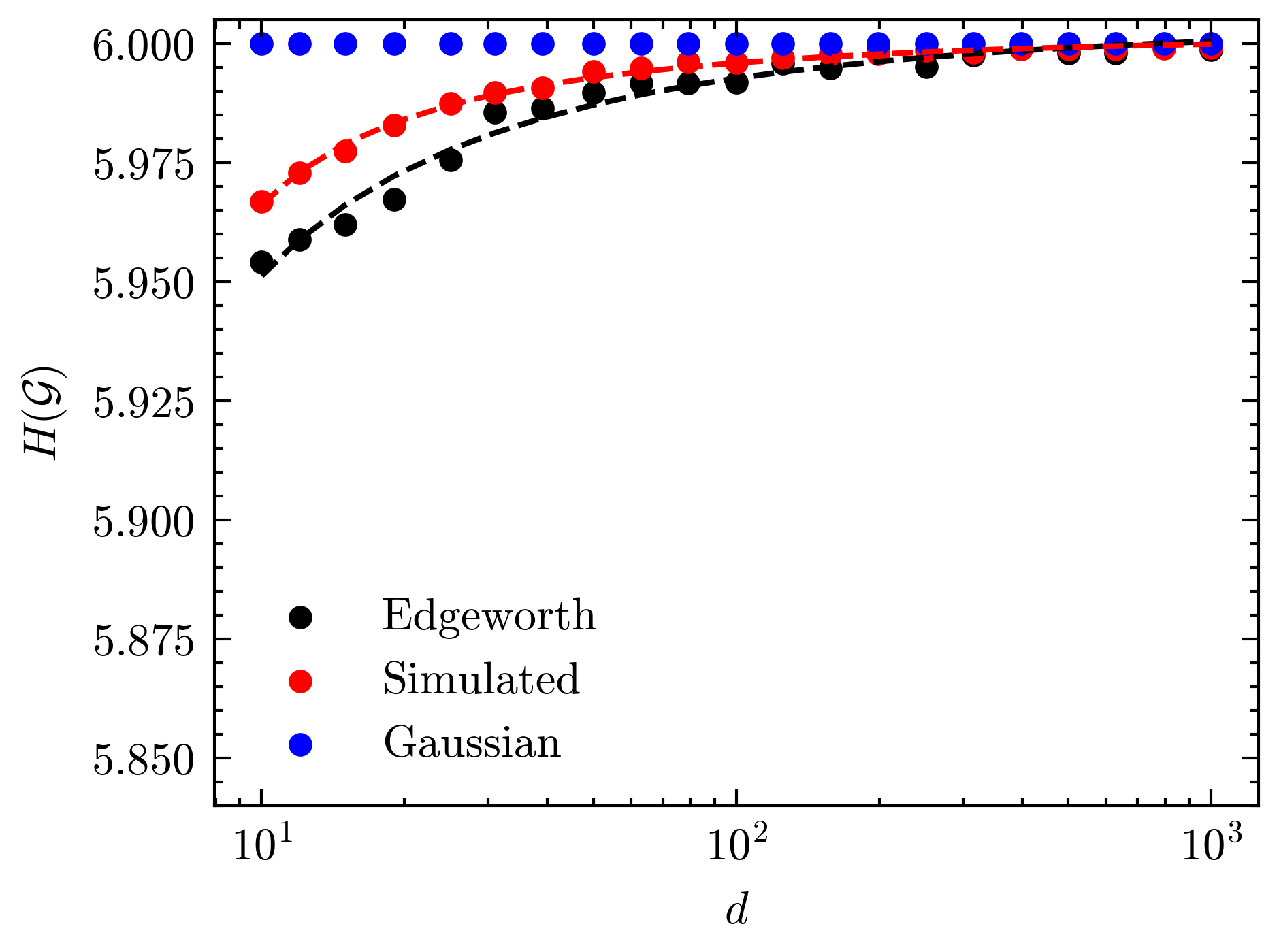}
    \caption{The entropy estimate for an ensemble of $4$-node hard RGGs with uniform distributed nodes in the $d$-dimensional hypercube (left) and hypertorus (right) given by the Gaussian and Edgeworth estimates against the simulated entropy. The dotted lines show a fit of the curve $H(\graph) = a-b(d^{-\frac{1}{2}}+c)$. The solid line is the entropy of the uniform distribution on $\binom{4}{2}=6$ variables, which is the theoretical maximum entropy.}
    \label{fig:entropy_estimate}
\end{figure}

\section{Conclusion}
In this work, we have investigated the distribution of random geometric graphs in high and low-dimensional cubes and tori, under the assumption that the nodes are distributed independently in each coordinate. In low dimensions, we calculated exactly the entropy of an RGG ensemble in closed form for the first time, and showed some numerical results about where the maximum of the entropy lies. We showed that, if the underlying geometry has periodic boundary conditions and the nodes are uniformly distributed, then the distribution of hard RGGs will converge to the Erd\H os-R\'enyi ensemble in distribution for a fixed number of nodes as the dimension tends to infinity. However, we proved that we do not see convergence to the ER ensemble when the nodes are not uniformly distributed. This challenges standard ideas about using RGGs to model high-dimensional data which is not usually assumed to be uniform. In the high-dimensional cube, provided the distribution of nodes has kurtosis greater than 1, adjacent edges are \textit{always} positively correlated, and so the limiting distribution is always different to the ER ensemble, and has lower entropy. In both geometries, \textit{soft} RGGs converge to the ER ensemble as $d\tti$, because soft connection functions remove the dependence on distance between nodes as $d\tti$. We then showed that the distribution of hard RGGs in the high-dimensional cube has a stationary point at the mean distance between nodes, and presented numerical evidence that this should be a global maximum in two specific cases, the uniform and truncated Gaussian distribution. Finally, we developed an Edgeworth correction to the central limit theorem for the uniform distribution of nodes in the $d$-cube and $d$-torus, allowing us to numerically approximate the scaling of RGG entropy in dimension. We showed that the entropy of a hard RGG ensemble in a finite dimension $d$ is an order $d^{-\frac{1}{2}}$ correction away from the entropy of the ensemble in infinite dimensions for both the cube and the torus. \newline

There are numerous possible extensions to this work. First, it would be interesting to see which other geometries converge to the ER limit, and which converge to different limits. In particular, different $L^p$ balls could be considered, and it could be useful to know whether correlations between adjacent distances are preserved in the $d\tti$ limit for any $p < \infty$ (in the sense that the unit cube can be viewed as the $L^{\infty}$ ball). This work also motivates the study of RGGs on the torus with non-uniformly distributed nodes. We may be able to develop an information-theoretic test for geometry when sampling RGGs with non-uniformly distributed nodes. For example, here we have shown that a `larger-than-expected' proportion of complete and empty graphs is indicative of underlying positive correlation.

\section*{Funding}
This work was supported by the EPSRC Centre for Doctoral Training in Computational Statistics and Data Science (COMPASS) to OB. 

\section*{Acknowledgments}
This work made use of the supercomputer BluePebble, one of the University of Bristol's high-performance computing facilities \url{https://www.bristol.ac.uk/acrc/high-performance-computing/}. We would also like to thank Kiril Bangachev for his helpful comments.

\appendix
\section*{Appendix}
\section{Exact Graph Probabilities}
\label{exact_calcs}
\subsection{Exact Graph Probabilities for the 3 node hard RGG in the 1D Torus}
As explained in Section 3, we will calculate the probabilities of a graph having $k$ edges, denoted by $p_k$. For brevity, we will write $\rho_T(x,y)$ as $|x-y|$. Assuming $r_0 < 1/3$ for now, we have
\begin{equation}
    p_3 = \int_{[0,1]^2} \mathbb{I}\{|z_2| <r_0\}\mathbb{I}\{|z_3|<r_0\}\mathbb{I}\{|z_2-z_3|<r_0\}dz_2dz_3 = 3r_0^2
\end{equation}
We can also re-centre the domain on 0 due to translation invariance, and change the bounds of integration to $[-1/2,1/2]^2$.
\begin{equation}
    p_3 = \int_{-r_0}^{r_0}\int_{-r_0}^{r_0}\mathbb{I}(|z_2-z_3| < r_0) dz_2dz_3
\end{equation}
\begin{equation}
    = \int_{-r_0}^{r_0} \int_{\max(-r_0, z_3-r_0)}^{\min(r_0, z_3+r_0)}dz_2dz_3
\end{equation}
\begin{equation}
    =2\int_0^{r_0} (2r_0 - z_3) dz_3 = 3r_0^2
\end{equation}
By employing a similar technique, we find
\begin{equation*}
    p_2 = \int_{[0,1]^3} \mathbb{I}\{|z_1-z_2|<r_0\}\mathbb{I}\{|z_1-z_3|<r_0\}\mathbb{I}\{|z_2-z_3|>r_0\}dz_1dz_2dz_3
\end{equation*}
\begin{equation}
    = \int_{-r_0}^{r_0}\int_{-r_0}^{r_0} (1-\mathbb{I}\{|z_2-z_3|<r_0\}) dz_2dz_3 = 4r_0^2 - p_3 = r_0^2
\end{equation}
\begin{equation}
    p_1 = \int_{[0,1]^3} \mathbb{I}(|z_1-z_2| < r_0)(1-\mathbb{I}(|z_1-z_3|<r_0))(1-\mathbb{I}(|z_2-z_3|<r_0))dz_1dz_2dz_3
\end{equation}
\begin{gather}
    \int_{[0,1]^2} \mathbb{I}\{|z_2|<r_0\} - \mathbb{I}\{|z_2|<r_0\}\mathbb{I}\{|z_2-z_3|<r_0\} \nonumber \\
    - \mathbb{I}\{|z_2|<r_0\}\mathbb{I}\{|z_3|<r_0\}\mathbb{I}\{|z_2-z_3|<r_0\}dz_2dz_3 + p_3
\end{gather}
\begin{equation}
    = 2r_0 - 5r_0^2
\end{equation}
\begin{equation}
    p_0 = 1 - p_3 - 3p_2 - 3p_1 = (1-3r_0)^2
\end{equation}
When $1/2 \geq r_0 > 1/3$, note that $p_0 = 0$. Using the same method as above, we have for $r_0 > 1/3$,
\begin{equation*}
    p_3 = 12r_0^2-6r_0+1
\end{equation*}
\begin{equation*}
    p_2 = 6r_0 - 1 - 8r_0^2
\end{equation*}
\begin{equation}
    p_1 = (2r_0-1)^2
\end{equation}

\subsection{Exact Graph Probabilities for the 3 node hard RGG in $[0,1]$}
Now $|\cdot|$ denotes the standard absolute value. Without loss of generality we may assume that $z_1 < z_2 < z_3$, then multiply the result by 6. To start we will set $r_0 < \frac{1}{2}$. 
\begin{equation}
    p_3 = 6\int_0^1\int_0^{z_3}\int_0^{z_2} \mathbb{I}\{z_1>z_2-r_0\}\mathbb{I}\{z_1>z_3-r_0\}\mathbb{I}\{z_2>z_3-r_0\}dz_1dz_2dz_3
\end{equation}
\begin{equation}
    = 6 \int_0^1 \int_0^{z_3} \mathbb{I}(z_2 > z_3-r_0)\int_{\max(0, z_3-r_0)}^{z_2}dz_1dz_2dz_3
\end{equation}
\begin{equation}
    = 6 \int_0^1 \int_{\max(0, z_3-r_0)}^{z_3} z_2 - \max(0, z_3-r_0) dz_2dz_3
\end{equation}
\begin{equation}
    = 6 \int_0^1 \frac{z_3^2}{2}+\frac{(\max(0,z_3-r_0))^2}{2} - z_3\max(0, z_3-r_0) dz_3
\end{equation}
\begin{equation}
     = 3r_0^2-2r_0^3
\end{equation}
Repeating for $p_0$, $p_1$ and $p_2$ we obtain
\begin{equation}
    p_2 = r_0^2 - \frac{4}{3}r_0^3
\end{equation}
\begin{equation}
    p_1 = \frac{14}{3}r_0^3 - 6r_0^2 + 2r_0
\end{equation}
\begin{equation}
    p_0 = (1-2r_0)^{3}
\end{equation}
The same calculations for $r_0 \geq 1/2$ give:
\begin{equation}
    p_3 = 3r_0^2 - 2r_0^3
\end{equation}
\begin{equation}
    p_2 = \frac{4}{3}r_0^3-\frac{1}{3}({3}r_0-1)^2
\end{equation}
\begin{equation}
    p_1 = \frac{2}{3}(1-r_0)^{3}
\end{equation}
\begin{equation}
    p_0 = 0
\end{equation}

\section{Numerical Approximation Procedure}
\label{error}

Here we describe in detail the numerical calculation of the entropy calculated in Table \ref{tab:max_entropy}. The overarching strategy is as follows. First, we approximately determine the region of $\hat{r}_0$, and then perform a grid search on this approximate region. We fit a quadratic to these points using least squares regression, and use standard techniques to estimate the covariance matrix of the coefficients of this fit. Then, we use the delta method for error propagation to estimate the error on the estimated maximum entropy value propagated by the error on the coefficients. \\

As a preliminary step, we first simulate an ensemble of $10^6$ graphs over the range $r_0 \in [0,D]$ with a step size of $D/50$. This allows us to approximately find the location of the maximum. Next, out of the intervals $I_i = [i/50, (i+1)/50]$, we find the one that contains the approximate maximum from the last step, say $I_{k}$.\\

We now take the interval $\hat{I} := I_{k-1} \cup I_k \cup I_{k+1}$, and perform another grid search. With $L := 10^8$, we simulate $L$ graphs at evenly spaced $r_0 = x_i \in \hat{I}$ for $i=1,...,N$ where $N=100$, and calculate the entropy of the resulting ensemble at each $r_0 = x_i$. At the end of the procedure, we check that the calculated maximum lies in $\hat{I}$, which was the case for each experiment we ran. \\

We then perform a quadratic least squares regression fit on the remaining points. This gives us an equation of the form $\tilde{y}_i = ax_i^2 + bx_i + c$ locally around the maximum. We find the value of the entropy maximiser, $\hat{r}_0$, as standard, 
\begin{equation}
    \hat{r}_0 = -\frac{b}{2a}
\end{equation}
and $\bar{p}_{max}$ is calculated using numerical integration of
\begin{equation}
    \bar{p}_{max} = \int_0^Df(r)p(r/\hat{r}_0)dr
\end{equation}
on which the error is less than $10^{-10}$. It remains to calculate the error of the approximation procedure. Using the results from \cite{roulston1999estimating}, we know that the variance of each entropy estimate $\tilde{y}_i$ is given by
\begin{equation}
    \sigma_i^2 = \frac{1}{L}\sum_{j=1}^{2^{\binom{n}{2}}} (\log(p_j) + \tilde{y}_i)^2p_j(1-p_j)
\end{equation}
for each $j$, we have that \begin{equation}
    (\log(p_j) + \tilde{y}_)^2p_j(1-p_j) \leq (\log(p_j) + \binom{n}{2})^2p_j(1-p_j)
\end{equation}
\begin{equation}
    \leq \binom{n}{2}^2 + 2\binom{n}{2}\log(p_j)p_j(1-p_j) + \log(p_j)^2p_j(1-p_j) 
\end{equation}
now it is simple to check that (e.g. by differentiating and finding the maximum) that $\log(p_j)p_j(1-p_j) \leq 1$ for each $p_j \in [0,1]$. So, a very crude upper bound on $\sigma_i^2$ is given by
\begin{equation}
    \sigma_i^2 \leq \frac{1}{L} 2^{\binom{n}{2}}\left(1+\binom{n}{2}\right)^2 = \frac{128}{L}
\end{equation}
when $n=3$. Now, we need an estimate on the error of the coefficients of the quadratic approximation. The error on the approximation is given by the residual sum of squares,
\begin{equation}
    RSS = \sum_{i=1}^{N} (\tilde{y}_i - y_i)^2 = \sum_{i=1}^N \epsilon_i^2
\end{equation}
where we assume that $\tilde{y}_i= y_i+\epsilon_i$, with $\epsilon_i \sim N(0,\sigma^2)$, and $y_i$ is the true value of the entropy evaluated at $r_0 = x_i$. Now, we know that
\begin{equation}
    \mathbb{E}[RSS] = \sum_{i=1}^N\mathbb{E}[\epsilon_i^2] = N\sigma^2 \leq \frac{128N}{L}
\end{equation}
Suppose $\vec{\theta} = (a,b,c)$ is the vector of estimate parameters. It is then standard that the covariance matrix of the parameters is given by
\begin{equation}
    Cov(\vec{\theta}) = \sigma^2(X^TX)^{-1}
\end{equation}
where $X$ is the $3\times N$ matrix, given by $X_{ij} = x_j^i$. Then, if we write $S_k := \sum_{i=1}^N x_i^k$,
\begin{equation}
    Cov(\vec{\theta}) = \sigma^2\left(\begin{matrix}
         S_4 &  S_3 & S_2 \\
         S_3 &  S_2 & S_1 \\
         S_2 & S_1 & N
    \end{matrix}\right)^{-1}
\end{equation}
The determinant of $(X^TX)$ is given by
\begin{equation}
    \det(X^TX) = S_4(NS_2-S_1^2) - S_3(NS_3 - S_1S_2) + S_2(S_3S_1-S_2^2)
\end{equation}
We will run into an issue here that, since the values of $x_i$ are very close to each other (generally all within 0.2 of each other, and 100 of these values) the conditioning number of this matrix is of the order $10^{17}$. To circumvent this issue, we standardise the $x_i$ values. Set $\tilde{S}_k$ to be the sum 
\begin{equation}
    \tilde{S}_k := \left(\frac{\sum_{i=1}^N (x_i - \mu_x)}{\sigma_x}\right)^k
\end{equation}
where $\mu_x$ and $\sigma_x$ are the mean and standard deviation of the $x_i$ respectively. We then perform the quadratic least squares regression on the rescaled variables, which are more spaced out, reducing the magnitude of the determinant. After re-scaling, we obtan the matrix
\begin{equation}
    \tilde{X^T}{\tilde{X}} = \left(\begin{matrix}
        \tilde{S}_4 & \tilde{S}_3 & \tilde{S}_2 \\
        \tilde{S}_3 & \tilde{S}_2 & \tilde{S}_1 \\
        \tilde{S}_2 & \tilde{S}_1 & N
    \end{matrix}\right)
\end{equation}
which has conditioning number of order $10^2$. To estimate the error on the rescaled coefficients $\tilde{\theta} = (\tilde{a}, \tilde{b}, \tilde{c})$, we use the delta method. This is justified because each of the calculated coefficients is a weighted sum of the $\tilde{y}_i$, and so, with $\hat{\theta}$ as the true coefficients of the quadratic expansion of the entropy, the quantity $\sqrt{N}(\tilde{\theta} - \hat{\theta})$ is approximately Gaussian. See e.g. \cite{bishop1975discrete}.

We need to calculate the variance on two estimates. The first is the rescaled $x_i$ that maximises the quadratic. Denote this by $\tilde{x}_{max}$. Then, we need to calculate the variance of the corresponding maximum entropy, given by $y_{max} = \tilde{a}\tilde{x}_{max}^2 + \tilde{b}\tilde{x}_{max} + \tilde{c}$. The delta method says that to estimate the variance of the unscaled $x_{max}$, we first note that the scaling means we can write the unscaled $x_{max}$ in terms of the scaled coefficients, and take the partial derivative in each coordinate to give the Jacobian
\begin{equation}
    J_{x_{max}} = \left(-\frac{\tilde{b}\sigma_x}{\tilde{a}^2}, \frac{\sigma}{\tilde{a}}, 0\right)
\end{equation}
then
\begin{equation}
    \label{eq:x_var}
    Var(x_{max}) \approx \sigma^2 J_{x_{max}}(\tilde{X}^T\tilde{X})^{-1}J_{x_{max}}^T \leq \frac{128N}{L}J_{x_{max}}(\tilde{X}^T\tilde{X})^{-1}J_{x_{max}}^T
\end{equation}
Similarly, for $y_{max}$, we get
\begin{equation}
    J_{y_{max}} = \left(\frac{\tilde{b}^2}{4\tilde{a}^2}, -\frac{\tilde{b}}{2\tilde{a}}, 1\right)
\end{equation}
so
\begin{equation}
    \label{eq:y_var}
    Var(y_{max}) \approx \sigma^2 J_{y_{max}}(\tilde{X}^T\tilde{X})^{-1}J_{y_{max}}^T \leq \frac{128N}{L}J_{y_{max}}(\tilde{X}^T\tilde{X})^{-1}J_{y_{max}}^T
\end{equation}
Using this gives us a method to estimate the standard error on the $\hat{r}_0$ (and therefore $\bar{p}_{max}$) and maximum entropy estimates, by taking the square root of the variance for $x_{max}$ and $y_{max}$ respectively. As an example, Figure \ref{fig:error_estimate} shows the quadratic fit for the $\eta=2$ Rayleigh fading SRGG in $\mathbb{T}^1$. The error estimate on the $\hat{r}_0$ value is $4.75\times 10^{-7}$ and the error on the maximum entropy estimate is $2.91\times 10^{-4}$. These are the square root of the upper bound in (\ref{eq:x_var}) and (\ref{eq:y_var}) respectively.
\begin{figure}
    \centering
    \includegraphics[width=0.8\linewidth]{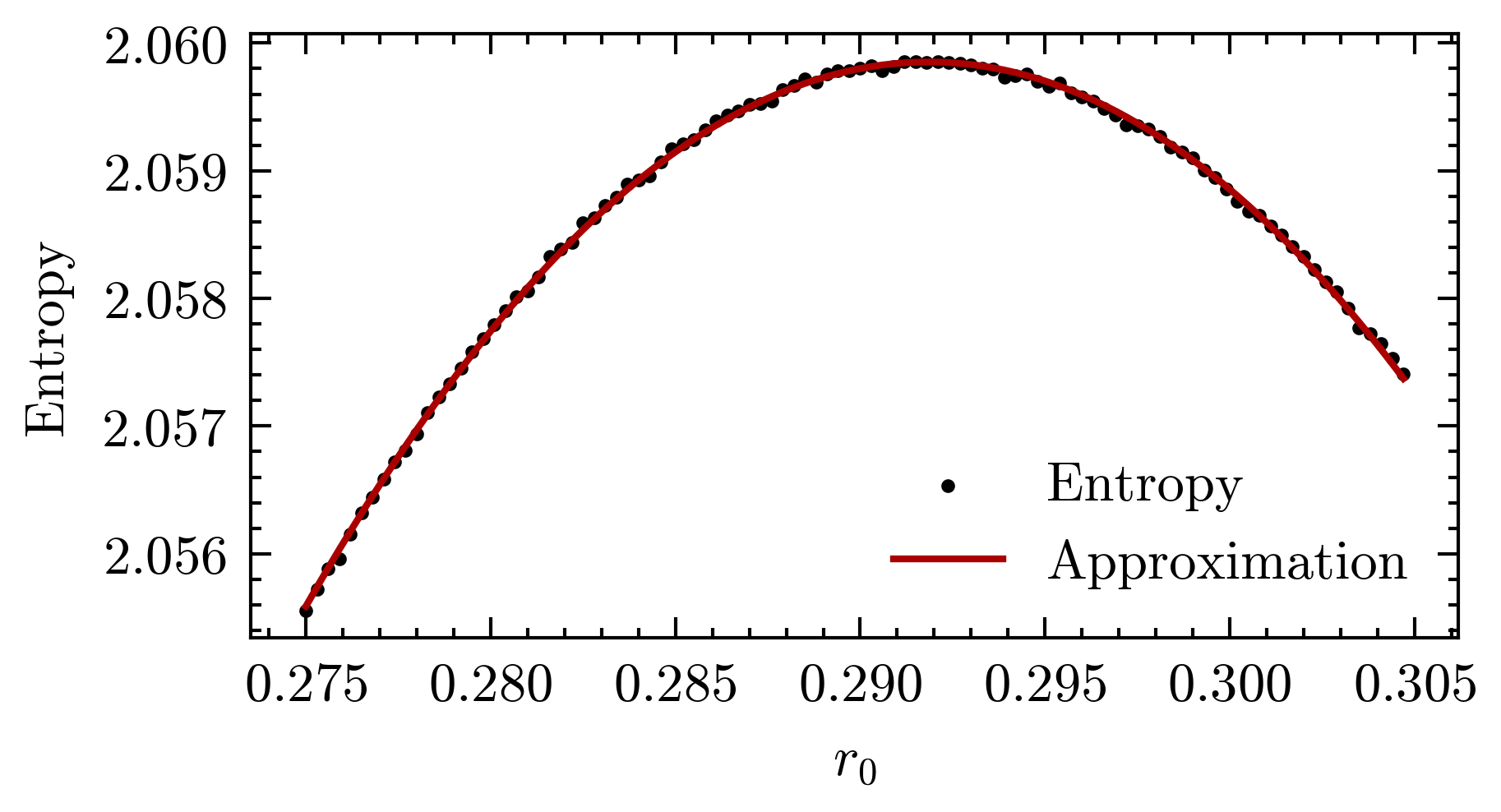}
    \caption{Example quadratic fit for the maximum entropy estimation. Here $L=10^8$, $N=100$ and the model is an $n=3$ node SRGG with a Rayleigh fading connection function with $\eta=2$ in $\mathbb{T}^1$.}
    \label{fig:error_estimate}
\end{figure}

\section{The Distribution with Kurtosis 1 is Bernoulli}
\label{sec:bernoulli}
Suppose $X$ is a real random variable with mean $\mu$, and kurtosis 1. Then by definition,
\begin{equation}
    \mathbb{E}[(X-\mu)^4] = \mathbb{E}[(X-\mu)^2]^2
\end{equation}
Let $Y=(X-\mu)^2$, then the above becomes
\begin{equation}
    \mathbb{E}[Y^2]-\mathbb{E}[Y]^2 = 0
\end{equation}
so $Y$ has 0 variance and is therefore almost surely constant. That is there exists a constant $c$ such that
\begin{equation}
    \prob{Y=c} = 1
\end{equation}
And therefore
\begin{equation}
    \prob{X \in \{\mu-c, \mu+c\}} = 1
\end{equation}
and therefore $X$ is a two-point (Bernoulli) random variable. It remains to show that $\prob{X = \mu-c} = \prob{X = \mu + c} = \frac{1}{2}$. Indeed if $\prob{X = \mu-c} = p = 1 - \prob{X=\mu+c}$, then
\begin{gather}
    \mathbb{E}[(X-\mu)^4] - \mathbb{E}[(X-\mu)^2]^2 = p^4(1-p) + p(1-p)^4 - p^4(1-p)^2 - 2p^3(1-p)^3 \nonumber\\ 
- (1-p)^4p^2
\end{gather}
\begin{equation}
    = p(1-p)(p^3 + (1-p)^3 - p^3(1-p) - 2p^2(1-p)^2 - p(1-p)^3) = 0
\end{equation}
so assuming $p \neq 0$ and $p \neq 1$, this implies
\begin{equation}
    p^4 - 2p^2(1-p)^2 + (1-p)^4 = 0
\end{equation}
\begin{equation}
    \Rightarrow (p^2-(1-p)^2)^2 = 0
\end{equation}
\begin{equation}
    \Rightarrow p = (1-p) = \frac{1}{2}
\end{equation}
so that $X$ is either Bernoulli with parameter $1/2$, or almost surely constant.

\section*{References}

\begingroup
    \renewcommand{\section}[2]{}

\endgroup
\end{document}